\documentclass[preprint,authoryear,10pt]{elsarticle}
\usepackage{amsfonts}
\usepackage{amsmath}
\usepackage{amssymb}
\usepackage{amsthm}
\usepackage[usenames,dvipsnames]{color}
\usepackage{natbib}{\tiny }
\usepackage{tikz}
\usepackage[utf8]{inputenc}
\usepackage{mathtools}
\usepackage{graphicx}
\usepackage{xfrac}
\usepackage{thmtools, thm-restate}
\usepackage{hyperref}
\usepackage{cleveref}
\usepackage{cancel}
\usepackage{mathtools}
\usepackage{bbm}
\usepackage{bm}
\usepackage[sf,bf,tiny]{titlesec}
\titlelabel{\thetitle. \enspace}

\theoremstyle{definition}
\newtheorem{thm}{Theorem}

\newtheorem{cor}{Corollary}

\newtheorem{rem}{Remark}
\title{\Large \bf Lancaster distributions and Markov chains with Multivariate Poisson-Charlier, Meixner and Hermite-Chebycheff polynomial \textcolor{black}{eigenfunctions}}

\author{\sc Robert Griffiths}

\begin{document}
\maketitle
%%%%%%%%%%%
\section*{Address}
\noindent
Professor R. C. Griffiths
\\Department of Statistics,\\ University of Oxford,\\ 1 South Parks Rd, Oxford, OX1 3TG, UK\\email:  griff@stats.ox.ac.uk.
\bigskip

\section*{Abstract} 
This paper studies new Lancaster characterizations of bivariate multivariate Poisson, negative binomial and normal distributions which have diagonal expansions in multivariate orthogonal polynomials.
The characterizations extend classical Lancaster characterizations of bivariate 1-dimensional distributions.
Multivariate Poisson-Charlier, Meixner and Hermite-Chebycheff orthogonal polynomials, used in the characterizations, are constructed from classical 1-dimensional orthogonal polynomials and multivariate Krawtchouk polynomials.
New classes of transition functions of discrete and continuous time Markov chains with these polynomials as eigenfunctions are characterized. The characterizations obtained belong to a class of mixtures of multi-type birth and death processes with fixed 
multivariate Poisson or multivariate negative binomial stationary distributions.
\medskip

\noindent
\emph{Keywords:}
Lancaster distributions; multivariate Krawtchouk polynomials; multivariate Poisson-Charlier polynomials; multivariate Meixner polynomials; multivariate Hermite-Chebycheff polynomials; multivariate birth and death processes.
%\\[0.3cm]
%\emph{Running Head:}
%multivariate Lancaster distributions.\\[0.3cm]
%%
%\emph{2010 MSC} Primary 60G99, 60E99; Secondary 42C05.

\section{Introduction}
%
%%%%%%%%%%%%%%%%%%%%%%%%%
A bivariate random vector $(X,Y)$ with marginal distributions $f(x)$ and $g(y)$ has a Lancaster probability distribution expansion $p(x,y)$ if 
\begin{equation}
p(x,y) = f(x)g(y)
\Bigl \{ 1 + \sum_{n\geq 1}\rho_nP_n(x)Q_n(y)\Bigr \},
\label{L:0}
\end{equation}
for $x$ and $y$ in the support of $X$ and $Y$, where $\{P_n(x)\}_{n=0}^\infty$, $\{Q_n(y)\}_{n=0}^\infty$
are orthonormal bases on $f$ and $g$, with $P_0=Q_0\equiv 1$. $\{\rho_n\}_{n=0}^\infty$ ($\rho_0 = 1$) is a correlation sequence between the two sets of orthogonal functions and satisfies
$
\mathbb{E}\big [P_m(X)Q_n(Y)\big ] = \delta_{mn}\rho_n.
$
The problem of characterizing correlation sequences $\{\rho_n\}$, for given orthonormal functions, such that a bilinear series  (\ref{L:0}) is a non-negative sum and thus a bivariate probability distribution has become known as a \emph{Lancaster problem} after Lancaster and colleagues who studied such bivariate distributions \citep{E1964,G1969,K1996,K1998,L1969,L1975}. Usually $f=g$ and $(X,Y)$ is exchangeable in these characterizations.
$X$ and $Y$ may be 1-dimensional or higher dimensional random variables.
\textcolor{black}{
The orthogonal polynomials used in these characterizations are from the Meixner class and have a generating function of the form
\begin{equation}
G(t,x) = h(t)e^{xu(t)}=\sum_{n=0}^\infty P_n(x)t^n/n!
\label{mclass:0}
\end{equation}
where $\{P_n(x)\}_{n=0}^\infty$ are orthogonal polynomials, $h(t)$ is a power series in $t$ with $h(0)=1$ and $u(t)$ is a power series with $u(0)=0$ and $u^\prime(0) \ne 0$. \citet{M1934} characterizes the class of weight functions and orthogonal polynomials with the generating function
(\ref{mclass:0}). They are: Binomial, Krawtchouk polynomials; Normal, Hermite-Chebycheff polynomials; Poisson, Poisson-Charlier polynomials; Poisson, Poisson-Charlier polynomials; Gamma, Laguerre polynomials; and lastly a hypergeometric weight function of the form
\[
|\Gamma(\lambda - ix)|^2e^{(2\phi - \pi)x}, \> -\infty < x < \infty,
\]
with Meixner-Pollaczek polynomials.
References to the Meixner class are \citet{M1934,E1964,L1975}. The Meixner polynomials are also included in the Sheffer polynomials, see for example \citet{S2000}.
}

%and the Krawtchouk polynomials.
%

%
%
A close connection exists between spectral expansions of transition functions of Markov chains and Lancaster expansions.
A discrete time reversible Markov chain can be constructed with transition functions 
\[
p(y\mid x) = f(y)
\Bigl \{ 1 + \sum_{n\geq 1}\rho_nP_n(x)P_n(y)\Bigr \}.
\]
The eigenvalues in this Markov chain are $\{\rho_n\}_{n=0}^\infty$, and eigenfunctions $\{P_n(y)\}_{n=0}^\infty$, satisfy
\[
\mathbb{E}\big [P_n(Y)\mid X=x\big ] = \rho_nP_n(x).
\]
The $k$-step transition functions are 
\[
p^{(k)}(y\mid x) = f(y)\Bigl \{ 1 + \sum_{n\geq 1}\rho_n^kP_n(x)P_n(y)\Bigr \}.
\]
A continuous time reversible Markov chain can be constructed as a Poisson imbedding of the discrete time Markov chain, whose transition functions are, for $t \geq 0$,
\begin{eqnarray}
&&\sum_{k=0}^\infty e^{-\lambda t}\frac{(\lambda t)^k}{k!}p^{(k)}(y\mid x)
\nonumber \\
&&=
f(y)\Bigl \{ 1 + \sum_{n\geq 1}e^{-\lambda t(1-\rho_n)}P_n(x)P_n(y)\Bigr \}.
\label{Poissonimbed:0}
\end{eqnarray}
\citet{B1954} shows that the most general continuous time reversible transition functions with stationary distribution $f(y)$ and eigenvectors $\{P_n(y)\}$
have eigenvalues of the form
\[
\rho_n(t) = e^{-\lambda t(1-\rho_n)}
\]
or a more general limit form obtained when $\lambda \to \infty$, $\rho_n \to 1$, with $\lambda (1-\rho_n)$ convergent. 

In the  gamma, Poisson, and negative binomial classes with exchangeability and orthogonal polynomial bases, where the random variables are non-negative, a necessary and sufficient condition that (\ref{L:0}) is a probability distribution is
\begin{equation}
\rho_n = \int_0^1z^n\varphi (dz),
\label{sequences}
\end{equation}
where $\varphi$ is a probability measure on $[0,1]$
\citep{G1969,K1996,K1998,S1968}.
The class of bivariate distributions is a convex set with extreme correlation sequence points $\{z^n\}$ for $z \in [0,1]$. $X$ and $Y$ are independent if $z=0$ and $X = Y$ if $z=1$.
The bivariate Poisson characterization is equivalent to 
\begin{equation}
X = U + V,\> Y=W + V,
\label{PREIC}
\end{equation}
where $U,V,W$ are independent Poisson random variables with means $\mu (1-Z)$, $\mu Z$, $\mu(1-Z)$ conditional on a random variable $Z$ with distribution $\varphi$ of (\ref{sequences}).
In the normal class 
\begin{equation}
\rho_n = \int_{-1}^1z^n\varphi (dz),
\label{nsequences}
\end{equation}
where $\varphi$ is a probability measure on $[-1,1]$ \citep{SB1967}. 
The extreme sequence when $z=0$ corresponds to independence of $X$ and $Y$ and $z=\pm1$ corresponds to $X=\pm Y$.
The characterization is equivalent to $(X,Y)$ having  a standard bivariate normal distribution with correlation $Z$, conditional on $Z$, where $Z$ has distribution $\varphi$ of (\ref{nsequences}).
For a general introduction to Lancaster expansions see \citet{L1969,K1996}. 
Bochner characterizations of continuous time Markov chains with polynomial eigenfunctions and stationary distributions which are Poisson, negative binomial, gamma and  normal are studied in \citet{G2009}.
\citet{E1969} characterized the correlation sequences $\{\rho_n\}^N_{n=0}$ in a bivariate binomial $(N,p\geq 1/2)$ distribution as having a representation
% \begin{equation*}
$\rho_n = \mathbb{E}\big [Q_n(Z)\big ]$, 
%\label{mixture:0}
%\end{equation*}
where $\{Q_n\}$ are Krawtchouk polynomials scaled so that $Q_n(0) = 1$ and $Z$ is a random variable on $0,1,\ldots,N$.
This characterization with $X$ having finite support is distinct from the moment characterizations (\ref{sequences}) and (\ref{nsequences}). It relies on a property of the Krawtchouk polynomials that there exists a random variable $W$ on $1,\ldots, N$ with distribution $\varphi_{xy}(w)$ such that for fixed $x,y$ and $r=0,\ldots, N$
\begin{equation}
Q_r(x)Q_r(y) = \mathbb{E}_{\varphi_{xy}}\Bigl [Q_r(W)\Bigr ].
\label{duplicate:0}
\end{equation}
\citet{DG2012} study the Eagleson characterization furthur and find another characterization in terms of bivariate binomial distributions where there are $N$ pairs of trials with random correlations. A nice connection is made with generalized Ehrenfest urn models.

Let $\{p_j\}_{j=1}^d$ be a discrete probability distribution on $d$ points. $\{u^{(l)}\}_{l=0}^{d-1} (\equiv \bm{u})$ will denote a complete set of orthogonal functions with
$u^{(0)} \equiv 1$, such that for $k,l=0,1,\ldots, d-1$, 
\begin{equation}
\sum_{j=1}^du_j^{(k)}u_j^{(l)}p_j = \delta_{kl}a_k.
\label{basic:0}
\end{equation}
This notation for orthogonal functions with a discrete state space follows that of Lancaster.
$\bm{u}$ satisfies a \emph{hypergroup} property if
\begin{equation}
u_j^{(r)}u_k^{(r)} = \sum_{l=1}^db_{jk}(l)u_l^{(r)},
\label{dhyper:0}
\end{equation}
where $\{b_{jk}(l)\}_{l=1}^d$ is a probability distribution for each $j,k = 1,\ldots ,d$. The formula (\ref{duplicate:0}) is a particular case satisfying the hypergroup property.
We supppose that the orthogonal functions can be scaled so that $u_d^{(r)}=1$, $r=0,\ldots ,d-1$. 
(The specific lower index $d$ is chosen for convenience only.) (\ref{dhyper:0}) is equivalent to 
\begin{equation} 
\mathfrak{s}(j,k,l) =\sum_{r=1}^da_r^{-1}u_j^{(r)}u_k^{(r)}u_l^{(r)}\geq 0,\>j,k,l=1,\ldots,d
\label{dhyper:1}
\end{equation}
where $b_{jk}(l)=p_l\mathfrak{s}(j,k,l)$. If an orthogonal basis forms a hypergroup then a Lancaster characterization is that
\begin{equation}
p_ip_j\big \{1 + \sum_{r=1}^{d-1}\rho_ra_r^{-1}u_i^{(r)}u_j^{(r)}\big \},\>i,j = 1,\ldots ,d
\label{dLanc:0}
\end{equation}
is non-negative and therefore a bivariate probability distribution if and only if
$\rho_r = \sum_{i=1}^dc_iu_i^{(r)}$ for a probability distribution $\{c_i\}_{i=1}^d$.
There is a general theory of orthogonal functions which have the hypergroup property for which see \citet{BH2008} and examples of orthogonal functions with this property in \citet{DG2014}.

In this paper the orthogonal polynomials in Lancaster expansions are  multivariate extensions of the 1-dimensional orthogonal polynomials in the Meixner class and Krawtchouk polynomials. \citet{G1971} and \citet{DG2014} construct multivariate Krawt\-chouck polynomials orthogonal on the multinomial distribution (\ref{multinomial:0}) and study their %hypergroup 
properties. Recent representations and derivations of orthogonality of these polynomials are in \citet{GVZ2013,GR2011,I2012,M2011}.
The multivariate Krawtchouk polynomials are eigenfunctions in classes of reversible composition Markov chains which have multinomial stationary distributions \citep{ZL2009}. \textcolor{black}{They also occur naturally in diagonalizing the joint distribution of marginals in a contingency table with a fixed number of entries (Section 2.2).
%, or a Poisson number of entries (Theorem 2).
} \citet{G1975,I2012a,GMVZ2014} construct and study multivariate Meixner polynomials, orthogonal on the multivariate Meixner distribution (\ref{Meixner:1}). \citet{GMVZ2014a} construct multivariate Poisson-Charlier polynomials on a $d$-dimensional Poisson product distribution. These polynomials are also mentioned briefy as limits from the multivariate Krawchouk and Meixner polynomials in \cite{G1971,G1975}. The three multivariate polynomial types are linked through the multivariate Krawtchouk polynomials. Authors emphasise different approaches to the multivariate orthogonal polynomials. \citeauthor{DG2014}'s approach is probabilistic and directed to Markov chain applications; \citeauthor{I2012}'s approach is via Lie groups; and \citeauthor{GVZ2013}'s physics approach is as matrix elements of group representations on oscillator states. 
\cite{X2013} studies discrete multivariate orthogonal polynomials which have a triangular construction of products of 1-dimensional orthogonal polynomials. These are particular cases of the polynomials studied in this paper \citep{DG2012}. 
% In this paper multivariate Poisson-Charlier, multivariate Meixner and multivariate Hermite-Chebycheff polynomials are considered in a probabilistic approach.

The multivariate Krawtchouk, Poisson-Charlier, Meixner and Hermite-Cheby\-cheff polynomials considered in this paper all have an elementary basis $\bm{u}$ in their construction; so notation for the polynomials includes $\bm{u}$ as a parameter, for example $Q_{\bm{n}}(\bm{X};\bm{u})$ for orthogonal polynomials on the multinomial. A simple way to think of the role $\bm{u}$ plays is that the orthogonal polynomials are constructed from powers of the linear forms $\{\sum_{j=1}^du^{(l)}_jX_j\}_{l=1}^{d-1}$. The index $\bm{n}$ refers to the highest degree in these linear forms, and there is only one term of highest degree $\bm{n}$.
\citet{DG2014} show that the multivariate Krawtchouk polynomials satisfy a hypergroup property if and only if their elementary basis $\bm{u}$ does. The multivariate Poisson-Charlier polynomials are orthogonal on the product distribution of $d$ 1-dimensional Poisson distributions with means $\bm{\mu}$. The elementary basis is orthogonal on $\bm{p}=\bm{\mu}/|\bm{\mu}|$. The multivariate Meixner polynomials are orthogonal on a distribution obtained as a multinomial mixture, with the number of trials having a negative binomial distribution.
The multivariate Hermite-Chebycheff polynomials are orthogonal on the product distribution of $d$ 1-dimensional Normal distributions with means $\bm{0}$ and variances $\bm{\tau}$.
The elementary basis $\bm{u}$ is orthogonal on $\bm{p}=\bm{\tau}/|\bm{\tau}|$. We use the notation that for a $d$-dimensional vector 
$\bm{z}$, $|\bm{z}| = z_1 + \cdots + z_d$. 
\textcolor{black}{
\citet{T1989,T1991} constructs multivariate orthogonal and biorthogonal polynomials on the multinomial, multivariate Poisson and Meixner distributions. These polynomials do not have a general basis $\bm{u}$ included in them.
}

This paper studies new Lancaster characterizations of bivariate multivariate Poisson, negative binomial and normal distributions and associated transition functions in Markov chains which have diagonal expansions in multivariate orthogonal polynomials.
%
%In this paper new Lancaster characterizations are obtained for bivarate distributions which have expansions in multivariate Poisson-Charlier, Meixner and Hermite-Chebycheff polynomials. 
The characterizations obtained in this paper have a similarity to the 1-dimensional ones, but are more difficult to obtain. In the Poisson case $\bm{X}$ and $\bm{Y}$ have a structure of random elements in common with random means, but fixed marginal means; in the Normal case $\bm{X}$ and $\bm{Y}$ have a random cross-correlation matrix, \textcolor{black}{ which is a new extension of the \citet{SB1967} characterization and of classical canonical correlation theory.} A hypergroup property of the elementary basis $\bm{u}$ plays an important role in the characterizations.
\textcolor{black}{
It is assumed in the characterizations of bivariate multivariate Poisson and Meixner distributions in Theorems 3,5 and in the associated remarks that follow. Characterization of the eigenvalues $\rho_{\bm{n}}$ in the expansion of bivariate multivariate Poisson and Meixner distributions in Theorems 3 and 5 is interesting in that the hypergroup property is used together with a moment characterization arising from the infinite support of the marginal distributions.
The hypergroup property is not needed in the construction of the polynomials or in Theorems 2,7. 
}

Transition functions of discrete and continuous time Markov chains related to the characterizations are constructed. The Markov chains are mixtures of birth and death chains in the Poisson and negative binomial models. Since Lancaster characterizations are equivalent to characterizations of transition functions with a given stationary distribution and eigenfunctions, the results of this paper have an importance in Markov chain theory. 

A minimal number of references helpful in understanding this paper are:
\citet{K1996} for Lancaster distributions; \citet{G2009} for the connection with transition functions of Markov chains; \citet{DG2014} for orthogonal polynomials on the multinomial distribution; and \citet{BH2008} for the hypergroup property.
\section{Multinomial distribution}
\subsection{Orthogonal polynomials on the multinomial distribution}
A brief description of the properties of orthogonal polynomials on the multinomial that are subsequently needed in the paper is now made. More details are in \citet{G1971} and \citet{DG2014}.
Define a collection of orthogonal polynomials
$\big \{Q_{\bm{n}}(\bm{X};\bm{u})\big \}$ with $\bm{n} = (n_1,\ldots n_{d-1})$
and $|\bm{n}| \leq N$ on the multinomial distribution
\begin{equation}
m(\bm{x};N,\bm{p}) = {N\choose \bm{x}}\prod_{j=1}^dp_j^{x_j},\>x_j \geq 0, \>j=1,\ldots, d,\> |\bm{x}|=N,
\label{multinomial:0}
\end{equation}
with $\{p_j\}_{j=1}^d$ a probability distribution,
as the coefficients of $w_1^{n_1}\cdots
w_{d-1}^{n_{d-1}}$ in the generating function
\begin{equation}
G(\bm{x},\bm{w}, \bm{u}) 
= \prod_{j=1}^d\Big (1 + \sum_{l=1}^{d-1}w_lu_j^{(l)}\Big )^{x_j},
\label{main_gf}
\end{equation}
where $\{u^{(l)}\}_{l=0}^{d-1}$ satisfies (\ref{basic:0}).
\textcolor{black}{
Often the parameters $N,\bm{p}$ are not included in the orthogonal polynomials notation, however sometimes for clarity in this paper they are included so 
$Q_{\bm{n}}(\bm{X};N,\bm{p},\bm{u})\equiv Q_{\bm{n}}(\bm{X};\bm{u})$.
}
It is straightfoward to show, by using the generating function, that
\begin{equation}
\mathbb{E}\Big [Q_{\bm{m}}(\bm{X}; \bm{u})Q_{\bm{n}}(\bm{X};\bm{u})\Big ] = \delta_{\bm{m}\bm{n}}{N\choose \bm{n},N-|\bm{n}|}.
\label{normalizing}
\end{equation}
Let $Z_1,\ldots ,Z_N$ be independent identically distributed random
variables such that
\[
P(Z=k) = p_k, \>k=1,\ldots ,d.
\]
Then with 
\[
X_i = |\{Z_k: Z_k=i, k=1,\ldots, d\}|,
\]
\begin{equation}
G(\bm{X},\bm{w}, \bm{u}) 
= \prod_{k=1}^N\Big (1 + \sum_{l=1}^{d-1}w_lu_{Z_k}^{(l)}\Big ).
\label{symfnrep}
\end{equation}
From (\ref{symfnrep})
\begin{equation}
Q_{\bm{n}}(\bm{X};\bm{u}) = \sum_{\{A_l\}}\prod_{k_1\in A_1}
u_{Z_{k_1}}^{(1)}
\cdots \prod_{k_{d-1}\in A_{d-1}}
u_{Z_{k_{d-1}}}^{(d-1)},
\label{partitionrep}
\end{equation}
where summation is over all partitions of subsets of $\{1,\ldots
,N\}$, $\{A_l\}$ such that $|A_l| = n_l$, $l = 1,\ldots ,d-1$. 
That is, the orthogonal polynomials are symmetrized orthogonal functions in the tensor product set
\textcolor{black}{
\[
\bigotimes_{k=1}^N\big \{1,u^{(i)}_{Z_k}\big \}_{i=1}^{d-1}
= \big \{u^{(i_1)}_{Z_1}u^{(i_2)}_{Z_2}\cdots u^{(i_N)}_{Z_N}\big \}_{i_1,\ldots, i_N=0}^{d-1},
\]
indexed by counting the number of elements in the set $\{i_1,\ldots,i_N\}$ which are $1,2,\ldots,d-1$. Recall that $u^{(0)}_{Z_l}=1$, $l=1,\ldots,N$. }
The orthogonal polynomials could equally well be defined by (\ref{partitionrep}) and the generating function (\ref{main_gf}) deduced.
% This is
%an analogue of the symmetric function representation (\ref{symrepa})
%for the Krawtchouck polynomials.
%
Let
\[
U_l = \sum_{k=1}^Nu_{Z_k}^{(l)} = \sum_{j=1}^du_j^{(l)}X_j
,\> l =1,\ldots ,d-1.
\]
$Q_{\bm{n}}(\bm{X};\bm{u})$ is a polynomial of degree $|\bm{n}|$ in $(U_1,\ldots ,U_{d-1})$
whose only term of maximal degree $|\bm{n}|$ is $\prod_1^{d-1}U_k^{n_k}$.
The construction of these polynomials is via a generating function approach (\ref{main_gf}) in \citet{G1971} and via symmetrized orthogonal functions (\ref{partitionrep}) in \citet{DG2014} rather than by a successive Gram-Schmidt orthogonalization approach on products of powers of 
$(U_1,\ldots ,U_{d-1})$. 
It is natural to wonder how the ordering of the polynomials by $\bm{n}$ would be in a Gram-Schmidt construction. 
The ordering is that a polynomial with index $\bm{n}$ preceeds a polynomial with index $\bm{n}^\prime$ if $|\bm{n}| < |\bm{n}^\prime|$, 
but if  $|\bm{n}| = |\bm{n}^\prime|$, the ordering in the Gram-Schmidt construction is arbitary because $Q_{\bm{n}}(\bm{X};\bm{u})$ is orthogonal to $\prod_1^{d-1}U_k^{n^\prime_k}$ if $\bm{n}\ne \bm{n}^\prime$. 

\textcolor{black}{
The multivariate Krawtchouk, Poisson, Meixner and Hermite-Chebycheff polynomials considered in this paper have a general unified construction and are unique within their classes for a given basis $\bm{u}$ as orthogonal polynomials in $(U_0,U_1,\ldots,U_{d-1})$ with single terms of highest degree. $U_0$ is the sum of the variables within a multivariate vector and does not play a role in the multivariate Krawtchouk polynomial degree. In the multivariate Poisson and Meixner distributions the conditional distribution of $\bm{X}$ given the sum of the elements, $|\bm{X}|$, has a multinomial distribution with parameter $N=|\bm{X}|$. The form of the orthogonal polynomials on these distributions is then a product of an orthogonal polynomial on the multinomial with $N=|\bm{X}|$ times an orthogonal polynomial on $|\bm{X}|$ which is either a 1-dimensional Poisson-Charlier polynomial or Meixner polynomial. There is a generality in these orthogonal polynomials, however the exact details of them need working through case by case. This is typical in papers dealing with the 1-dimensional Meixner class.
}

\subsection{Marginals in a contingency table}
The multivariate Krawtchouk polynomials diagonalize the joint distribution of marginal counts in a contingency table. Suppose $N$ observations are placed independently into an $r\times c$ table ($r \leq c$) with the probability of an observation falling in cell $(i,j)$ being $p_{ij}$. Denote the marginal distributions as $p_i^r = \sum_{j=1}^cp_{ij}$ and $p_j^c = \sum_{i=1}^rp_{ij}$. Let $p_{ij}$ have a Lancaster expansion (which is always possible, even for non-exchangeable $p_{ij}$)
\begin{equation}
p_{ij} = p^r_ip^c_j\big \{1 + 
\sum_{k=1}^{r-1}\rho_k u^{(k)}_iv^{(k)}_j\big \},
\label{elLancaster:0}
\end{equation}
where $\bm{u}$ and $\bm{v}$ are orthonormal function sets on $\bm{p}^r$ and $\bm{p}^c$. $\bm{u}$ is an orthonormal basis and if $r < c$ complete $\{v^{(k)}\}_{k=0}^{r-1}$ to a basis $\{v^{(k)}\}_{k=0}^{c-1}$.  Let $N_{ij}$ be the number of observations falling into cell $(i,j)$ and 
$X_i = \sum_{j=1}^cN_{ij}$, $Y_j = \sum_{i=1}^rN_{ij}$. $\bm{X}$ and $\bm{Y}$ are the marginal counts. Then
\textcolor{black}{
\begin{eqnarray}
&&P\big (\bm{X}=\bm{x},\bm{Y}=\bm{y}\big ) = 
%P\big (\bm{X}=\bm{x})P\big (\bm{Y}=\bm{y}\big )\\
m(\bm{x};N,\bm{p}^r)m(\bm{y};N,\bm{p}^c)\nonumber \\
&&~~\times
\Big \{1 + \sum_{\bm{n}}\rho_1^{n_1}\ldots \rho_{r-1}^{n_{r-1}}{N\choose \bm{n}}^{-1}Q_{\bm{n}}(\bm{x};N,\bm{p}^r,\bm{u})Q_{\bm{n}^\ast}(\bm{y};N,\bm{p}^c,\bm{v})\Big \},
\nonumber \\
\label{contingency:0}
\end{eqnarray}
}
where $\bm{n}^\ast = (\bm{n},\bm{0}_{c-r})$. \citet{AG1935} showed (\ref{contingency:0}) for a $2\times 2$ table with orthogonal polynomials the usual 1-dimensional Krawtchouk polynomials and \citet{G1971} for $r\times c$ tables.
%%%%%
%%%%%
\section{Multivariate Poisson}
\subsection{Multivariate Poisson-Charlier polynomials}
Let $X$ be a Poisson random variable with mean $\lambda$.  The Poisson-Charlier orthogonal polynomials on $X$, $\{C_n(X;\lambda)\}_{n=0}^\infty$,  have a generating function
\begin{equation}
\sum_{n=0}^\infty C_n(X;\lambda)\frac{z^n}{n!}
 = e^z\Big (1-\frac{z}{\lambda}\Big )^X.
 \label{PC:0}
\end{equation}
The orthogonality relationship is
\[
\mathbb{E}\big [C_m(X;\lambda)C_n(X;\lambda)\big ] = \delta_{mn}\frac{m!}{\lambda^m}.
\]
Let $\bm{X}=(X_1,\ldots ,X_d)$ be independent Poisson random variables with means $\bm{\mu}=(\mu_1,\ldots ,\mu_d)$. The probability distribution of $\bm{X}$ is
\begin{equation}
P(\bm{x};\bm{\mu}) = \prod_{i=1}^de^{-\mu_i}\frac{\mu_i^{x_i}}{x_i!}.
\label{MVPD:0}
\end{equation}
A natural set of multivariate Poisson-Charlier orthogonal polynomials is the product set $\bigotimes_{j=1}^d\big \{C_{n_j}(X_j;\mu_j) \big \}$. Another set is constructed by noting that the sum $|\bm{X}|$ is a Poisson random variable with mean $|\bm{\mu}|$ and the conditional distribution of
$\bm{X}$ given $|\bm{X}|$ is multinomial with parameters $(N=|\bm{X}|,\bm{p}=\bm{\mu}/|\bm{\mu}|)$. The next theorem is a new probabilistic construction.
\begin{thm}\label{PCDEF:0}
A set of multivariate Poisson-Charlier polynomials orthogonal on $\bm{X}$ is
\begin{equation}
C_{\bm{n}}(\bm{X};\bm{\mu},\bm{u}) = 
{(n_0)!}^{-1}C_{n_0}(|\bm{X}|- |\bm{n_1}|;|\bm{\mu}|)
|\bm{\mu}|^{-|\bm{n_1}|}Q_{\bm{n}_1}(\bm{X};|\bm{X}|,\bm{p},\bm{u}),
\label{MVPC:0}
\end{equation}
where $\bm{n} \in \mathbb{Z}_+^d$, $\bm{n}_1 = (n_1,\ldots, n_{d-1})$, and
$Q_{\bm{n}_1}(\bm{X};|\bm{X}|,\bm{p},\bm{u})$ are multivariate Krawtchouk polynomials, orthogonal on the multinomial random variable $\bm{X}$ given $|\bm{X}|$. $\bm{u}$ is an orthogonal basis satisfying (\ref{basic:0}).
The orthogonality relationship is
\begin{equation}
\mathbb{E}\Big [C_{\bm{m}}(\bm{X};\bm{\mu},\bm{u})C_{\bm{n}}(\bm{X};\bm{\mu},\bm{u})\Big ]
= \delta_{\bm{m}\bm{n}}|\bm{\mu}|^{-|\bm{n}|}\prod_{j=0}^{d-1}\frac{a_j^{n_j}}{n_j!}
\coloneqq  \delta_{\bm{m}\bm{n}}h_{\bm{n}}(|\bm{\mu}|,\bm{u})^{-1}.
\label{MVPC:2}
\end{equation}
The polynomials (\ref{MVPC:0}) are generated by the coefficients of
$w_0^{n_0}\cdots w_{d-1}^{n_{d-1}}$ in 
\begin{equation}
G_{\text{PC}}(\bm{X},\bm{w},\bm{u})
= e^{w_0}\prod_{i=1}^d\Big (1 - |\bm{\mu}|^{-1}w_0 + |\bm{\mu}|^{-1}\sum_{j=1}^{d-1}u_i^{(j)}w_j\Big )^{X_i}.
\label{MVPC:1}
\end{equation}
\end{thm}
% % %
% % %
\begin{proof}
To show that the generating function (\ref{MVPC:1}) is correct it is sufficient to show that it generates the orthogonal polynomials (\ref{MVPC:0}) as coefficients  of $w_0^{n_0}\cdots w_{d-1}^{n_{d-1}}$. The coefficient of 
$w_1^{n_1}\cdots w_{d-1}^{n_{d-1}}$ in  (\ref{MVPC:1}) is
\begin{equation}
e^{w_0}\Big (1-|\bm{\mu}|^{-1}w_0\Big )^{|\bm{X}|-|\bm{n}_1|}
|\bm{\mu}|^{-|\bm{n}_1|}Q_{\bm{n}_1}(\bm{X};|\bm{X}|,\bm{p},\bm{u})
\label{intermediate:0}
\end{equation}
so (\ref{MVPC:0}) holds from (\ref{intermediate:0}) and (\ref{PC:0}).
Orthogonality of the polynomials (\ref{MVPC:0}) can be shown by direct argument or using the generating function (\ref{MVPC:1}).
\begin{eqnarray}
&&\mathbb{E}\big [G_{\text{PC}}(\bm{X},\bm{z},\bm{u})G_{\text{PC}}(\bm{X},\bm{w},\bm{u})\big ]
\nonumber \\
&&=
\exp \Bigg \{
z_0+w_0+
\sum_{i=1}^d\mu_i\Big [
\Big (1 - |\bm{\mu}|^{-1}z_0 + |\bm{\mu}|^{-1}\sum_{j=1}^{d-1}u_i^{(j)}z_j\Big )
\nonumber \\
&&~~~~~~~~~~~~~~~~~~~~~~~~~~~~~\times\Big (1 - |\bm{\mu}|^{-1}w_0 + |\bm{\mu}|^{-1}\sum_{j=1}^{d-1}u_i^{(j)}w_j\Big )-1\Big ]\Bigg \}
\nonumber \\
&&=\exp \Big \{|\bm{\mu}|^{-1}\sum_{j=0}^{d-1}a_jz_jw_j\Big \}.
\label{Orthogcalc:0}
\end{eqnarray}
Now (\ref{MVPC:2}) follows as the coefficient of
 $z_0^{m_0}\cdots z_{d-1}^{m_{d-1}}w_0^{n_0}\cdots w_{d-1}^{n_{d-1}}$ in (\ref{Orthogcalc:0}).
\end{proof}
\begin{cor}\textcolor{black} {The transform of the multivariate Poisson-Charlier polynomials is
\begin{eqnarray}
C^*_{\bm{n}}(\bm{s};\bm{\mu},\bm{u})
&:=&
\mathbb{E}\Big [\prod_{j=1}^ds_j^{X_j}
C_{\bm{n}}(\bm{X};\bm{\mu},\bm{u}) \Big ]
\nonumber \\
&=&
\Big (\prod_{j=0}^{d-1}n_j!\Big )^{-1}
e^{\sum_{i=1}^d\mu_i(s_i-1)}
\Big ( 1 - S_0\Big )^{n_0}
\prod_{j=1}^{d-1}S_j^{n_j},
\label{transform:0}
\end{eqnarray}
where $S_j = \sum_{i=1}^dp_is_iu_i^{(j)}$.}
\end{cor}
%
% % %
\begin{proof}
\begin{eqnarray}
&&\mathbb{E}\Big [\prod_{j=1}^ds_j^{X_j}
G_{\text{PC}}(\bm{X},\bm{w},\bm{u})
\Big ]
\nonumber \\
&&= \exp \Bigg \{
w_0 + \sum_{i=1}^d\mu_i\Big [s_i\Big (1 - |\bm{\mu}|^{-1}w_0+|\bm{\mu}|^{-1}\sum_{j=1}^{d-1}u_i^{(j)}w_j\Big )-1\Big ]\Bigg \}
\nonumber \\
&&=
\exp \Bigg \{\sum_{i=1}^d\mu_i(s_i-1)
+w_0(1-S_0) + \sum_{j=1}^{d-1}S_jw_j 
\Bigg \}
\nonumber \\
\label{transform:1a}
\end{eqnarray}
Now (\ref{transform:0}) is the coefficient of $w_0^{n_0}\cdots w_{d-1}^{n_{d-1}}$ in (\ref{transform:1a}).
\end{proof}
\begin{rem}
A representation of a Poisson random vector $\bm{X}$ with independent entries and mean $\bm{\mu}$ is that
\[
\bm{X} = \bm{e}_{Z_1}+\bm{e}_{Z_2}+\cdots + \bm{e}_{Z_{|\bm{X}|}},
\]
where $\bm{e}_k$ is the $k$th unit vector and $\{Z_i\}_{i=1}^\infty$ is a sequence of independent random variables such that $P(Z_i = k) = p_k = \mu_k/|\bm{\mu}|$, $k = 1,\ldots,d$ and $|\bm{X}|$ is a Poisson random variable independent of the sequence with mean $|\bm{\mu}|$.
A symmetrized tensor product set, similar to (\ref{partitionrep})  with a random number of trials is then
\begin{eqnarray}
C_{\bm{n}}(\bm{X};\bm{\mu},\bm{u})
&=& {(n_0)!}^{-1}C_{n_0}(|\bm{X}|- |\bm{n_1}|;|\bm{\mu}|)
|\bm{\mu}|^{-|\bm{n_1}|}
\nonumber \\
&&\times \sum_{\{A_l\}}\prod_{k_1\in A_1}
u_{Z_{k_1}}^{(1)}
\cdots \prod_{k_{d-1}\in A_{d-1}}
u_{Z_{k_{d-1}}}^{(d-1)},
\label{Poisson:sym:0}
\end{eqnarray}
where summation is over all partitions of subsets of $\{1,\ldots
,|\bm{X}|\}$, $\{A_l\}$ such that $|A_l| = n_l$, $l = 1,\ldots ,d-1$. 
$C_{\bm{n}}(\bm{X};\bm{\mu},\bm{u})$ is a polynomial in the sums
$U_j=\sum_{i=1}^du^{(j)}_iX_i$ with a single leading term of degree $|\bm{n}|$ proportional to $U_0^{n_0}\cdots U_{d-1}^{n_{d-1}}$.
\end{rem}
\begin{rem}
It seems natural to include a constant function $u^{(0)}=1$ in the basis, thus including the Poisson-Charlier polynomials in $|\bm{X}|$ in the polynomial set, however it is possible to choose a basis $\bm{u}$ without a constant function. Then a generating function for the multivariate polynomials is
\begin{equation}
\exp \Big \{-\sum_{j=0}^{d-1}\sum_{i=1}^du_i^{(j)}p_iw_j\Big \}\prod_{i=1}^d\Big (1 + |\bm{\mu}|^{-1}\sum_{j=0}^{d-1}u_i^{(j)}w_j\Big )^{X_i}.
\label{general:gf:0}
\end{equation}
This generating function is in Section 11 of \citet{GMVZ2014a}, with a different notation.
\end{rem}
\begin{rem}
The polynomials $\{C_{\bm{n}}(\bm{X};\bm{\mu},\bm{u})\}$ can be obtained as a limit from a set of multivariate Krawtchouk polynomials. \cite{GMVZ2014a} show such a limit, which is also very briefly mentioned in \citet{G1971}. Let $\bm{X}^{(N)}$ be a multinomial $(N,\bm{p}^{(N)})$ random variable of dimension $d+1$ with 
$(p^{(N)}_i)_{i=1}^d = (N+|\bm{\mu}|)^{-1}\bm{\mu}$ and 
$p_{d+1}^{(N)} = (N+|\bm{\mu}|)^{-1}N$. Then from a classical limit theorem $(X_i^{(N)})_{i=1}^d$ converges to a Poisson random vector $\bm{X}$ with mean $\bm{\mu}$. 
Choose an orthogonal basis for the multivariate Krawtchouk polynomials of $\{v^{(j)}(N)\}_{j=0}^d$ with $v^{(0)}(N)=1$, and
\[
v_k^{(j)}(N) = \begin{cases}
|\bm{\mu}|^{-1} u^{(j)}_k,&k=1,\ldots d\\
 0,&k=d+1
 \end{cases}
\]
where $\{u^{(j)}\}$ is an orthogonal basis on $\bm{\mu}$.
Take
\[
v_k^{(d)}(N) =
\begin{cases}
-|\bm{\mu}|^{-1},&k=1,\ldots , d\\
N^{-1},&k=d+1.
\end{cases}
\]
The generating function for the multivariate Krawtchouk polynomials is
\begin{eqnarray}
&&\prod_{j=1}^{d+1}\Big (1 + \sum_{l=1}^{d}w_lv_j^{(l)}(N)\Big )^{x_j}
\nonumber \\
&&= 
\prod_{j=1}^{d}\Big (1 + \sum_{l=1}^{d-1}w_l|\bm{\mu}|^{-1}u_j^{(l)} - |\bm{\mu}|^{-1}w_{d}\Big )^{x_j}
\Big (1 + N^{-1}w_d\Big )^{N-\sum_1^dx_i}
\nonumber \\
&&\to
e^{w_d}\prod_{j=1}^{d}\Big (1 - |\bm{\mu}|^{-1}w_{d} +|\bm{\mu}|^{-1}\sum_{l=1}^{d-1}w_lu_j^{(l)} \Big )^{x_j}.
\label{plimit:0}
\end{eqnarray}
Change the indexing by setting $w_d$ to $w_0$, then (\ref{plimit:0}) is identical to (\ref{MVPC:1}). This shows that multivariate Krawtchouk polynomials with basis $\{v^{(l)}(N)\}$ converge to the multivariate Poisson-Charlier polynomials with basis 
$\bm{u}$.
\end{rem}
\subsection{Marginals in a Poisson array}
% %
The distribution of the marginals in a Poisson array has a diagonal expansion in the multivariate Poisson-Charlier polynomials as a natural extension from a contingency table. The next theorem is new.
\begin{thm}\label{Thm 2}
Let $(X_{ij})$, $i=1,\ldots ,r$, $j=1,\ldots ,c$, $r \leq c$, be an $r\times c$ array of independent Poisson random variables with means $\bm{\mu} = (\mu_{ij})$.
 Let $p_{ij}= \mu_{ij}/|\bm{\mu}|$. Suppose there is a Lancaster expansion for $i=1,\ldots ,r$, $j=1,\ldots ,c$,
 \begin{equation}
 p_{ij} = p^r_{i\cdot}p^c_{\cdot j}
 \Big \{ 1 + \sum_{k=1}^{r-1}\rho_ku^{(k)}_iv^{(k)}_j\Big \}
 \label{Lancaster:0}
 \end{equation}
 where $\bm{u}$ and $\bm{v}$ are orthonormal functions on $\bm{p}^r$ and $\bm{p}^c$ respectively. 
 Denote the marginals of the table by 
 $\bm{X} = (X_{i\cdot})_{i=1}^r$,
  $\bm{Y} = (Y_{\cdot j})_{j=1}^c$ and their means by $\bm{\mu}_{\bm{x}}$, $\bm{\mu}_{\bm{y}}$.
 Then
 \begin{eqnarray}
 &&P\big (\bm{X}=\bm{x},\bm{Y}=\bm{y}\big ) =
 P(\bm{x};\bm{\mu}) P(\bm{y};\bm{\mu})
%P\big (\bm{X}=\bm{x}\big )P\big (\bm{Y}=\bm{y}\big )
 \nonumber \\
 &&~~~
 \times
  \Big \{
 1 + \sum_{\bm{n};|\bm{n}|\geq 1}
 \rho_1^{n_1}\cdots \rho_r^{n_{r-1}}h_{\bm{n}}(|\bm{\mu}|)
C_{\bm{n}}(\bm{x};\bm{\mu}_{\bm{x}},\bm{u})
C_{\bm{n}^*}(\bm{y};\bm{\mu}_{\bm{y}},\bm{v})
\Big \},
\nonumber \\
\label{Table:1}
 \end{eqnarray}
 where $\bm{n} \in \mathbb{Z}_+^r$,
 $\bm{n}^* = (\bm{n},\bm{0}_{c-r})\in \mathbb{Z}_+^c$ and $h_{\bm{n}}(|\bm{\mu}|) = |\bm{\mu}|^{|\bm{n}|}\prod_{j=0}^{d-1}n_j!$.
 \end{thm}
 % % %
 \begin{proof}
The conditional distribution of $(X_{ij})$ given $|(X_{ij})|$ is that of an $r\times c$ contingency table with $|(X_{ij})|$ observations. The conditional distribution of the marginals $(\bm{X},\bm{Y})$ therefore has a Lancaster expansion (\ref{contingency:0}), where $N=|(X_{ij})|=|\bm{X}|=|\bm{Y}|$.
Then 
\begin{eqnarray}
&&\mathbb{E}\Big [C_{\bm{m}}(\bm{X};\bm{\mu}_{\bm{x}},\bm{u})C_{\bm{n}}(\bm{Y};\bm{\mu}_{\bm{y}},\bm{v})\Big ]
\nonumber \\
&&= \mathbb{E}\Bigg [{(m_0)!}^{-1}C_{m_0}(|\bm{X}|- |\bm{m_1}|;|\bm{\mu}|)
{(n_0)!}^{-1}C_{n_0}(|\bm{Y}|- |\bm{n_1}|;|\bm{\mu}|)
\nonumber \\
&&~~\times\mathbb{E}\Big [
Q_{\bm{m}_1}(\bm{X};|\bm{X}|,\bm{p}^r,\bm{u})
Q_{\bm{n}_1}(\bm{Y};|\bm{Y}|,\bm{p}^c,\bm{v}) \>\Big |\> |(X_{ij})| \Big ] \Bigg ]
\nonumber \\
&&=\mathbb{E}\Bigg [C_{m_0}(|\bm{X}|- |\bm{n}_1|;|\bm{\mu}|)
C_{n_0}(|\bm{X}|- |\bm{n}_1|;|\bm{\mu}|){|\bm{X}|\choose \bm{n}_1}\Bigg ]
\nonumber \\
&&~~~~~\times\delta_{\bm{m}_1\bm{n}_1^*}
{(m_0)!}^{-1} {(n_0)!}^{-1}|\bm{\mu}|^{-(|\bm{m}_1|+|\bm{n}_1|)}\rho_1^{n_1}\cdots \rho_{r-1}^{n_{r-1}}
\nonumber \\
&&=
{(n_0)!}^{-1}\mathbb{E}\Bigg [C_{m_0}(|\bm{X}|;|\bm{\mu}|)
C_{n_0}(|\bm{X}|;|\bm{\mu}|)\Bigg ]
\nonumber \\
&&~~~~~\times\delta_{\bm{m}_1\bm{n}_1^*}(n_0!n_1!\cdots n_{r-1}!)^{-1}|\bm{\mu}|^{-|\bm{m}_1|}\rho_1^{n_1}\cdots \rho_{r-1}^{n_{r-1}}
\nonumber \\
&&=\delta_{\bm{m}\bm{n}^*}(n_0!n_1!\cdots n_{r-1}!)^{-1}|\bm{\mu}|^{-|\bm{n}|}\rho_1^{n_1}\cdots \rho_{r-1}^{n_{r-1}}
\nonumber \\
&&=\delta_{\bm{m}\bm{n}^*}h_{\bm{n}}(|\bm{\mu}|)^{-1}\rho_1^{n_1}\cdots \rho_{r-1}^{n_{r-1}}
\label{Table:0}
\end{eqnarray}
% % %
(\ref{Table:1}) now follows from (\ref{Table:0}).
\end{proof}
\subsection{Multivariate Poisson Lancaster expansions}
The next theorem is the most important in this paper, together with the identification, in Remark \ref{PoissonMC}, of a Markov chain which has transition functions of $P(\bm{Y}=\bm{y}\mid \bm{X}=\bm{x})$ when $(\bm{X},\bm{Y})$ has a probability distribution (\ref{H:0}).
\begin{thm}\label{HG:2}Let $\bm{u}$ be an orthogonal basis on $\bm{p}=\bm{\mu}/|\bm{\mu}|$ with (\ref{basic:0}) holding. Suppose the basis satisfies the hypergroup property (\ref{dhyper:1}).
Then
\begin{equation}
P(\bm{x};\bm{\mu}) P(\bm{y};\bm{\mu})\Big \{
1 + \sum_{|\bm{n}|\geq 1}\rho_{\bm{n}}h_{\bm{n}}(|\bm{\mu}|,\bm{u})
C_{\bm{n}}(\bm{x};\bm{\mu},\bm{u})
C_{\bm{n}}(\bm{y};\bm{\mu},\bm{u}) \Big \}
\label{H:0}
\end{equation}
is non-negative and thus a proper bivariate Poisson distribution if and only if
\begin{eqnarray}
\rho_{\bm{n}} = 
\mathbb{E}_{\bm{\xi}}\Big [
\prod_{j=0}^{d-1}\rho_j(\bm{\xi})^{n_j}\Big ],
\label{H:1}
\end{eqnarray}
where
$\rho_j(\bm{\xi})=\sum_{i=1}^{d-1}u_i^{(j)}\xi_i$
and $\bm{\xi}$ is a random vector in the unit simplex.
\end{thm}
\begin{proof}
\emph{Necessity.}  Suppose (\ref{H:0}) holds. Let $(\bm{X},\bm{Y})$ be a pair of Poisson random vectors with distribution (\ref{H:0}).  Then
\begin{equation}
\rho_{\bm{n}}C_{\bm{n}}(\bm{x};\bm{\mu},\bm{u})
= \mathbb{E}\Big [C_{\bm{n}}(\bm{Y};\bm{\mu},\bm{u}) \mid \bm{X}=\bm{x}\Big ].
\label{cc:0}
\end{equation}
$C_{\bm{n}}(\bm{x};\bm{\mu},\bm{u})$ has only one leading term of degree $|\bm{n}|$, proportional to 
\[
m(\bm{n},\bm{x},\bm{u})
= \prod_{j=0}^{d-1}\Bigg (\sum_{i=1}^{d-1}u_i^{(j)}x_i\Bigg )^{n_j}.
\]
Rearranging (\ref{cc:0})
\begin{equation}
\mathbb{E}\Big [m(\bm{n},\bm{Y},\bm{u})\mid \bm{X}=x\Big ]
=\rho_{\bm{n}}m(\bm{n},\bm{x},\bm{u})
+ R_{|\bm{n}|-1}(\bm{x})
\label{H:3}
\end{equation}
where $R_{|\bm{n}|-1}(\bm{x})$ is a polynomial of degree $|\bm{n}|-1$ in $\bm{x}$. Divide (\ref{H:3}) by $|\bm{x}|^{|\bm{n}|}$ and let $|\bm{x}|\to \infty$ such that $x_d/|\bm{x}| \to 1$ and 
$x_j/|\bm{x}| \to 0$, $1 \leq j <d$.
Let $\bm{\xi}$ be a random variable with the limit distribution of $\bm{Y}/|\bm{x}|$ given 
$\bm{X} = \bm{x}$.
The limit is taken through a sub-sequence such that the limit distribution is proper.
Recalling that $u_d^{(j)}=1$ for $j=0,\ldots ,d$, (\ref{H:1}) holds from the limit. $\rho_0(\bm{\xi}) = |\bm{\xi}|$ and with a particular $\bm{n}^\prime = (n, 0,\ldots,0)$, 
$
\rho_{\bm{n}^\prime} 
= \mathbb{E}_{\bm{\xi}}\big [|\bm{\xi}|^n\big ]
$.
Since $-1 \leq \rho_{\bm{n}^\prime} \leq 1$ for all such $\bm{n}^\prime$,  $|\bm{\xi}|$ must belong to the unit simplex.
\medskip

\noindent
\emph{Sufficiency.} Suppose that (\ref{H:1}) holds. 
Let  $S_j=\sum_{i=1}^dp_is_iu_i^{(j)}$, $j=0,\ldots,d-1$
and similarly for $T_j$. Note that $S_0 = \sum_{i=1}^dp_is_i$, $T_0 = \sum_{i=1}^dp_it_i$.
The potential probability generating function (\emph{pgf}) of a bivariate distribution $(\bm{X},\bm{Y})$ formed from (\ref{H:0}) is
\begin{eqnarray}
%\mathbb{E}\Big [\prod_{i,j=1}^ds_i^{X_i}t_j^{Y_j}\Big ]
&&\sum_{\bm{n}:|\bm{n}| \geq 0}
\rho_{\bm{n}}h_{\bm{n}}(|\bm{\mu}|,\bm{u})
C^*_{\bm{n}}(\bm{s};\bm{\mu},\bm{u})
C^*_{\bm{n}}(\bm{t};\bm{\mu},\bm{u})
~~~~
\nonumber \\
&&=\exp \big \{|\bm{\mu}|(S_0-1 +T_0-1)  \big \}
\nonumber \\
&&\times
\mathbb{E}_{\bm{\xi}}\Big [\exp \Big \{
|\bm{\mu}||\bm{\xi}|(1-S_0)(1-T_0)
+|\bm{\mu}|\sum_{k=1}^{d-1}a_k^{-1}\rho_k(\bm{\xi})S_kT_k
%\nonumber \\
%&&~~~~~~~~~~
\Big \}\Big ]
\nonumber \\
&&=\exp \big \{|\bm{\mu}|(S_0-1 +T_0 - 1)  \big \}
\nonumber \\
&&\times
\mathbb{E}\Big [\exp \Big \{
|\bm{\mu}||\bm{\xi}|\big [(1-S_0)(1-T_0)-S_0T_0\big ]
+|\bm{\mu}|\sum_{k=0}^{d-1}a_k^{-1}\rho_k(\bm{\xi})S_kT_k
%\nonumber \\
%&&~~~~~~~~~~
\Big \}\Big ]
\nonumber \\
&&=\mathbb{E}_{\bm{\xi}}\Big [\exp \big \{|\bm{\mu}|(1-|\bm{\xi}|)
%\sum_{i=1}^dp_i(s_i-1+t_i-1)
(S_0-1+T_0-1)
\big \}
\nonumber \\
&&~~~~\times
\exp \Big \{
|\bm{\mu}|\sum_{j,k,l=1}^dp_jp_k\xi_l\mathfrak{s}(j,k,l)(s_jt_k-1)
\Big \}\Big ]
\label{H:4}
\end{eqnarray}
The identity $\sum_{j,k=1}^dp_jp_k\mathfrak{s}(j,k,l) = 1$, for $l=1,\ldots,d$ is used in the last line.
(\ref{H:4}) is a \emph{pgf} because
$\mathfrak{s}(j,k,l) \geq 0$, completing the sufficiency proof.
\end{proof}
\begin{rem} $\bm{X}$ and $\bm{Y}$ are independent if
$\rho_{\bm{n}} = 0$ for all $|\bm{n}| \geq 1$. 
$\bm{X}=\bm{Y}$ if $\rho_{\bm{n}} = 1$ for $|\bm{n}| \geq 1$,
achieved when $\xi_1=\cdots = \xi_{d-1}=0$ and $\xi_d=1$.
\end{rem}
\begin{rem}
The structure of $(\bm{X},\bm{Y})$ as marginals of a Poisson array with random means is clear from (\ref{H:4}).
\begin{equation}
X_j =^{\cal D} Z_j + \sum_{k=1}^dZ_{jk},\>
Y_k =^{\cal D} Z_k^\prime + \sum_{j=1}^dZ_{jk}.
\label{REIC:0}
\end{equation}
where $\bm{Z}$, $\bm{Z}^\prime$ and $(Z_{jk})$ are Poisson random variables, conditionally independent given $\bm{\xi}$, with conditional means 
\[\mathbb{E}\big [\bm{Z}\mid \bm{\xi}\big ]
= \mathbb{E}\big [\bm{Z}^\prime\mid \bm{\xi}\big ]
= (1-|\bm{\xi}|)\bm{\mu}
\]
 and
\begin{eqnarray}
\mathbb{E}\big [Z_{jk}\mid \bm{\xi}\big ]
&=& 
|\bm{\mu}|p_jp_k\sum_{l=1}^d\xi_l\mathfrak{s}(j,k,l)
\nonumber \\
&=& 
|\bm{\mu}||\bm{\xi}|
p_jp_k\Big \{1 + \sum_{r=1}^{d-1}\theta_ra_i^{-1}u_j^{(r)}u_k^{(r)}\Big \}
\nonumber \\
&=& |\bm{\mu}||\bm{\xi}|p_{jk}(\bm{\xi}),
\label{Lrandom}
\end{eqnarray}
where $\theta_r = \sum_{l=1}^d\omega_lu_l^{(r)}$, with $\omega_l = \xi_l/|\bm{\xi}|$ and 
\begin{equation}
p_{jk}(\bm{\xi})=p_jp_k\Big \{1 + \sum_{r=1}^{d-1}\theta_ra_r^{-1}u_j^{(r)}u_k^{(r)}\Big \}
\label{Lrandom:0}
\end{equation}
is a bivariate probability distribution.
(\ref{Lrandom}) and (\ref{Lrandom:0}) are non-negative because of the hypergroup property satisfied by the  orthogonal basis.

If $|\bm{\xi}|\equiv 1$ then the structure is that of marginal distributions in a Poisson table as in Theorem \ref{Thm 2}.

The structure (\ref{REIC:0}) is an extension of the bivariate 1-dimensional 
Poisson characterization (\ref{PREIC}) mentioned in the Introduction.

\end{rem}
% % %
\begin{rem}
The conditional \emph{pgf} of the extreme point distribution $\bm{Y}\mid \bm{X} = \bm{x}$ is
\begin{eqnarray}
&&\exp \Big \{|\bm{\mu}|(1-|\bm{\xi}|)(T_0-1)\Big \}
\prod_{i=1}^d
\Big (1 - |\bm{\xi}| + \sum_{j,l=1}^d\xi_lp_j\mathfrak{s}(i,j,l)t_j
\Big )^{x_i}
\nonumber \\
&&=
\exp \Big \{|\bm{\mu}|(1-|\bm{\xi}|)(T_0-1)\Big \}
\prod_{i=1}^d
\Big (1 - |\bm{\xi}| + |\bm{\xi}|\sum_{j=1}^dp_{j\mid i}(\bm{\omega})t_j
\Big )^{x_i},
\label{CondPoisson:0}
\end{eqnarray}
where
$
p_{j\mid i}(\bm{\omega})= p_{ij}(\bm{\omega})/p_i
$
is a conditional probability distribution.
Both factors in (\ref{CondPoisson:0}) are \emph{pgfs}.
\end{rem}
%%%
\begin{rem}\label{PoissonMC} An interpretation of the extreme points in the conditional distribution of $\bm{Y}\mid \bm{X}$ with a \emph{pgf} 
(\ref{CondPoisson:0}) is as the transition functions in a discrete time Markov chain. Consider a queueing system with an infinite number of servers and Poisson arrivals. The servers are of $d$ possible types. Customers arrive at rate $|\bm{\mu}|(1-|\bm{\xi}|)$ and choose an available server of type $i$ with probability $p_i$. Each queue has length zero or one. At a time epoch each of the $X_i$ customers in type $i$ queues is either served with probability $1-|\bm{\xi}|$ or changes to an available server of type $j$ with probability $|\bm{\xi}|p_{j\mid i}(\bm{\omega})$ and is not served in the current epoch.
This describes the transition from the numbers $\bm{X}$ in the queues of $d$ types to $\bm{Y}$. The stationary distribution of the Markov chain is Poisson with mean $\bm{\mu}$. The general transition functions are a mixture of extreme point transition functions over a distribution for $\bm{\xi}$. The importance of this representation is that it characterizes \emph{every} discrete time Markov chain with stationary Poisson $(\bm{\mu})$ distribution and eigenfunctions $\{C_{\bm{n}}(\cdot;\bm{\mu},\bm{u})\}$ as being a mixture of these extreme point processes. 

The extreme points of the 1-dimensional Markov chain of the total number in the queue $|\bm{X}|$, without regard to type, have a structure that is a classical $M/M/\infty$ queue with birth rates $|\bm{\mu}|(1-|\bm{\xi}|)$, death rates $1-|\bm{\xi}|$ and a stationary Poisson distribution with mean $|\bm{\mu}|$.
\end{rem}
\begin{rem}The distribution of $(\bm{X},\bm{Y})$ is asymptotically normal as $\bm{\mu} \to \bm{\infty}$, independent within vectors, with means and variances $\bm{\mu}$, and random cross covariance matrix
\[
|\bm{\mu}||\bm{\xi}|p_{jk}(\omega).
\]
\end{rem}
\begin{rem}\label{ContinuousPoisson}
The general form of transition functions from $\bm{x}$ to $\bm{y}$ in time $t$ in a continuous time reversible Markov process $\{\bm{X}(t)\}_{t \geq 0}$ with a stationary multivariate Poisson distribution and multivariate Poisson-Charlier polynomial eigenfunctions has an expansion
\begin{equation}
P(\bm{y};\bm{\mu})\Big \{
1 + \sum_{|\bm{n}|\geq 1}\rho_{\bm{n}}(t)h_{\bm{n}}(|\bm{\mu}|,\bm{u})
C_{\bm{n}}(\bm{x};\bm{\mu},\bm{u})
C_{\bm{n}}(\bm{y};\bm{\mu},\bm{u}) \Big \}
\label{Hcts:0}
\end{equation}
where 
\begin{eqnarray}
\rho_{\bm{n}}(t) &=& \exp \Big \{-\lambda t \Big [1-\rho_{\bm{n}}\Big ]
\Big \}
\nonumber \\
&=& 
\exp \Big \{-\lambda t \mathbb{E}_{\bm{\xi}}\Big [1-\prod_{j=0}^{d-1}
\rho_j(\bm{\xi})^{n_j}\Big ]
\Big \},
\label{Hcts:1}
\end{eqnarray}
or a limit form, from \cite{B1954}.

We consider a simpler particular case which leads to a multi-type birth and death population process with immigration, deaths, and changes of type. Recall the notation $\rho_k(\bm{\xi}) = \sum_{i=1}^du_i^{(k)}\xi_i$.
The idea is to choose $\lambda,\bm{\xi}$ so that $\rho_j(\bm{\xi}) \to 1$ and
\begin{equation}
\lambda \mathbb{E}_{\bm{\xi}}\Big [1-\prod_{j=0}^{d-1}
\rho_j(\bm{\xi})^{n_j}\Big ] \approx \lambda \sum_{j=0}^{d-1}n_j\big (1-\rho_j(\mathbb{E}[{\bm{\xi}}])\big ).
\label{small:0}
\end{equation}
To achieve this set $\mathbb{E}[{\xi}_i] = \lambda^{-1}\gamma_i$, $i=1,\ldots,d-1$ and
$1-\mathbb{E}[\xi_d] = \sum_{i=1}^{d-1}\lambda^{-1}\gamma_i + \lambda^{-1}\nu$,
where $\bm{\gamma},\nu \geq 0$.
Then $1-\rho_0(\mathbb{E}[\bm{\xi}]) = 1 - \mathbb{E}[|\bm{\xi}|] = \lambda^{-1}\nu$ and recalling that
$u_d^{(j)}=1$,
\begin{eqnarray*}
1-\rho_j(\mathbb{E}[\bm{\xi}]) &= & 1 - \sum_{i=1}^{d-1}u^{(j)}_i\mathbb{E}[\xi_i] -\mathbb{E}[\xi_d]
\nonumber \\
&=& \lambda^{-1}\big (\nu + \sum_{i=1}^{d-1}\gamma_i(1-u_i^{(j)})\big ).
\end{eqnarray*}
Assume also that $\lambda \text{Var}\big (\rho_k(\bm{\xi})\big ) \to 0$.
Because of the hypergroup property of $\bm{u}$, $|u_i^{(j)}| \leq 1$, so terms in the sum are non-negative. Let $\lambda \to \infty$. The eigenvalues have the limit form
\begin{equation}
\rho_{\bm{n}}(t) = \exp \Big \{-t 
|\bm{n}|\nu + \sum_{j=1}^{d-1}n_j\theta_j(\bm{\gamma})
\Big \},
\label{Hcts:2}
\end{equation}
where $\theta_j(\bm{\gamma})=\sum_{i=1}^{d-1}\gamma_i(1-u_i^{(j)})$.
Now consider the \textcolor{black}{infinitesimal} \emph{pgf} for transitions in $(t,t+\Delta t)$ as $\Delta t \to 0$, that is,
$\mathbb{E}\Big [\prod_{k=1}^dt_k^{X_k(t+\Delta t)}\mid \bm{X}(t) = \bm{x}\Big ]$.
Only the linear terms in $\rho_{\bm{n}}(\Delta t)$ need to be considered.
\begin{eqnarray*}
\rho_{\bm{n}}(\Delta t) &=& 
\exp \Big \{-\Delta t\Big (
|\bm{n}|{\nu} + \sum_{j=1}^{d-1}n_j\theta_j({\bm{\gamma}})\Big )
\Big \}
\nonumber \\
%\label{Hcts:2}
\end{eqnarray*}
where ${\nu} = \mathbb{E}[{\nu}]$, ${\bm{\gamma}}= \mathbb{E}[\bm{\gamma}]$ and
\begin{eqnarray}
\widetilde{\rho}_k(\mathbb{E}[\bm{\xi}]) &\coloneqq & \exp \Big \{-\Delta t\big ({\nu} + \theta_k({\bm{\gamma}})\big ) \Big \}
\nonumber \\
&\approx &  1 - \Delta t \big ({\nu} + \theta_k({\bm{\gamma}})\big ).
\label{evcts}
\end{eqnarray}
The infintesimable \emph{pgf} is now obtained by substituting the eigenvalues (\ref{evcts}) in (\ref{CondPoisson:0}). Calculating the quantities appearing:
\[
\mathbb{E}[1-|\bm{\xi}|]= 1 - \widetilde{\rho}_0(\mathbb{E}[\bm{\xi}]) \approx 1 - e^{-{\nu} \Delta t} \approx {\nu} \Delta t,
\]
and
\begin{eqnarray}
\mathbb{E}[|\bm{\xi}|p_{j\mid i}(\bm{\omega})] 
&\approx & p_j\sum_{k=0}^{d-1}\Big (1 - \Delta t\big ({\nu} + \theta_k({\bm{\gamma}})\big )\Big )a_k^{-1}u_i^{(k)}u_j^{(k)}
\nonumber \\
&=& \delta_{ij}\big (1 - \Delta t({\nu} + |{\bm{\gamma}}|)\big )
 + \Delta tp_j\sum_{l=1}^{d-1}\gamma_l\mathfrak{s}(i,j,l)
\nonumber \\
&=&\delta_{ij}\big (1 - \Delta t({\nu} + |{\bm{\gamma}}|)\big )
 + \Delta t |{\bm{\gamma}}| \widetilde{p}_{j\mid i},
\label{calculation:0}
\end{eqnarray}
where
\[
\widetilde{p}_{j\mid i}\coloneqq 
p_j\sum_{l=1}^{d-1}\big ({\gamma}_l/|{\bm{\gamma}}|\big )\mathfrak{s}(i,j,l)
\]
is a transition matrix with stationary distribution $\bm{p}$.
Then (\ref{CondPoisson:0}) is, to order $\Delta t$,
\begin{eqnarray}
&&\Big ( 1 - \Delta t|\bm{\mu}|{\nu} \sum_{i=1}^dp_it_i\Big )
\nonumber \\
&&~\times 
\prod_{i=1}^d\Big (
\big ( 1 - \Delta t({\nu} + |{\bm{\gamma}}|\big )t_i + \Delta t{\nu}
 + \Delta t \sum_{j=1}^d|{\bm{\gamma}}|\widetilde{p}_{j\mid i}t_j\Big )^{x_i}.
 \label{ctspgf:0}
\end{eqnarray}
$\{\bm{X}(t)\}_{t\geq 0}$ is now identified as a multi-type birth and death process from (\ref{ctspgf:0}). Of course only the terms to order $\Delta t$ matter, but the \emph{pgf} is clearer in the current form. Thinking of $\bm{X}(t)$ as counts of individuals in a population at time $t$; immigration occurs at rates ${\nu}|\bm{\mu}|p_i$; deaths occur to individuals at rate ${\nu}$; and changes of type from $i$ to $j$ occur at rates ${\bm{\gamma}}\widetilde{p}_{j\mid i}$. The stationary distribution is Poisson $(\bm{\mu})$. $|\bm{X(t)}|$ is a conventional birth and death process with birth rates when $|\bm{X}(t)|=|\bm{x}|$ of $|\bm{\mu}|{\nu}$ and death rates $|\bm{x}|{\nu}$. The spectral expansion for the transition functions of $|\bm{X}(t)|$ in terms of the 1-dimensional Poisson-Charlier polynomials is well known, see for example \citet{S2000} p33.
\end{rem}

\section{Multivariate negative binomial}
\subsection{Multivariate Meixner polynomials}
Let $X$ be a negative binomial random variable with probability distribution
\begin{equation}
\frac{\Gamma (\alpha + x)}{\Gamma (\alpha)x!}\frac{\theta^x}{(1+\theta)^{x+\alpha}},\>x = 0,1,\ldots,\>\alpha,\theta > 0.
\label{Meixner:0}
\end{equation}
The Meixner orthogonal polynomials on $X$,
$\big \{M_n(X;\alpha,\gamma)\big \}_{n=0}^\infty$ have a generating function
\begin{equation}
\sum_{n=0}^\infty \frac{\Gamma (\alpha + n)}{\Gamma (\alpha)n!}M_n(x;\alpha,\kappa)z^n = (1-z/\kappa)^x(1-z)^{-(x+\alpha)},
\label{Meixner:1}
\end{equation}
with $\kappa = \theta/(1+\theta)$.
The orthogonality relationship for the 1-dimensional Meixner polynomials is
\begin{equation}
\mathbb{E}\big [M_m(X;\alpha,\kappa)M_n(X;\alpha,\kappa)\big ]
= \delta_{mn}\frac{\Gamma (\alpha)n!}{\Gamma (\alpha + n)\kappa^n}.
\label{onedorthog}
\end{equation}
The multivariate Meixner distribution is
\begin{equation}
P(\bm{x};\alpha,\theta,\bm{p}) = \frac{\Gamma (\alpha+|\bm{x}|)}{\Gamma (\alpha)|\bm{x}|!}
\frac{\theta^{|\bm{x}|}}{(1+\theta)^{\alpha+|\bm{x}|}}
{|\bm{x}|\choose \bm{x}}\prod_{i=1}^dp_i^{x_i},
\label{Meixner:4}
\end{equation}
for $\bm{x} \in \mathbb{Z}_+^d$.
The conditional distribution of $\bm{X}$ given $|\bm{X}|$ is 
$m(\bm{x};|\bm{x}|,\bm{p})$.
The \emph{pgf} of (\ref{Meixner:4}) is
\begin{equation}
\Big [ 1 - \theta \sum_{i=1}^dp_i(s_i-1)\Big ]^{-\alpha}.
\end{equation}
Let
\[
P(\bm{x};\mu,\bm{p}) = 
\prod_{i=1}^de^{-\mu p_i}
\frac{(\mu p_i)^{x_i}}
{x_i!}.
\]
The multivariate Meixner distribution is a Poisson-Gamma mixture
\begin{equation}
P(\bm{x};\alpha,\theta,\bm{p})
= \mathbb{E}_\mu\big [P(\bm{x};\mu,\bm{p})\big ],
\label{Meixner:5}
\end{equation}
where $\mu \sim$ Gamma ($\alpha,\theta$) with density
\[
\frac{\mu^{\alpha - 1}}{\theta^\alpha
\Gamma (\alpha)}e^{-\mu/\theta},\>\mu > 0.
\]
\textcolor{black}{
\citet{G1975}, \citet{I2012a} and \citet{GMVZ2014} construct multivariate Meixner polynomials of the form in the next Theorem.
}
\begin{thm}\label{Def:M}
Let $\bm{X}$ have a multivariate Meixner distribution $P(\bm{x};\alpha,\theta,\bm{p})$. A set of multivariate Meixner polynomials, orthogonal on $\bm{X}$ is 
\begin{eqnarray}
&&M_{\bm{n}}(\bm{X};\alpha,\theta,\bm{p},\bm{u})
\nonumber\\
&&=\frac{\Gamma (\alpha + n_0)}{\Gamma (\alpha)n_0!}M_{n_0}(|\bm{X}|-|\bm{n}_1|;\alpha, \theta/(1+\theta))Q_{\bm{n}_1}(\bm{X};|\bm{X}|,\bm{p},\bm{u}),
\label{Meixner:7}
\end{eqnarray}
where $\bm{n} \in \mathbb{Z}_+^d$, $\bm{n}_1 = (n_1,\ldots, n_{d-1})$, and
$Q_{\bm{n}_1}(\bm{X};|\bm{X}|,\bm{p},\bm{u})$ are multivariate Krawtchouk polynomials, orthogonal on the multinomial random variable $\bm{X}$ given $|\bm{X}|$. $\bm{u}$ is an orthogonal basis satisfying (\ref{basic:0}).
The orthogonality relationship is
\begin{eqnarray}
&&\mathbb{E}\Big [
M_{\bm{m}}(\bm{X};\alpha,\theta,\bm{p},\bm{u})
M_{\bm{n}}(\bm{X};\alpha,\theta,\bm{p},\bm{u})
\Big ]
\nonumber \\
&&~~= \delta_{\bm{m}\bm{n}}
\frac{\Gamma(|\bm{n}|+\alpha)}
{\Gamma(\alpha)n_0!\cdots n_{d-1}!}\Bigg (\frac{1+\theta}{\theta}\Bigg )^{n_0}\theta^{|\bm{n}|-n_0}
\prod_{j=1}^{d-1}a_j^{n_j}.
\label{Mei:0}
\end{eqnarray}
The multivariate Meixner polynomials are coefficients of
$w_0^{n_0}\ldots w_{d-1}^{n_{d-1}}$ in the generating function  \citep{G1975}
\begin{equation}
G_{\text{M}}(\bm{X},\bm{w},\alpha,\theta,\bm{u}) = (1-w_0)^{-(|\bm{X}|+\alpha)}\prod_{i=1}^d
\Big (1 - w_0\frac{1+\theta}{\theta} + \sum_{j=1}^{d-1}u_i^{(j)}w_j\Big )^{X_i}.
\label{Meixner:6}
\end{equation}
These polynomials have a representation as a mixture of the multivariate Poisson-Charlier polynomials:
\begin{eqnarray}
&&P(\bm{X};\alpha,\theta,\bm{p})
\Bigg (\frac{\theta}{1+\theta}\Bigg )^{n_0}
M_{\bm{n}}(\bm{X};\alpha,\theta,\bm{p},\bm{u})
\nonumber \\
&&~~=
\mathbb{E}_\mu\big [P(\bm{X};\mu,\bm{p})\mu^{|\bm{n}|}C_{\bm{n}}(\bm{X};\bm{\mu},\bm{u})\big ],
\label{Meixner:6}
\end{eqnarray}
where $\mu \equiv |\bm{\mu}|$.
\end{thm}
\begin{cor}\label{Corr:M}
The transform of the multivariate Meixner polynomials is
\begin{eqnarray}
M^*_{\bm{n}}(\bm{s};\alpha,\theta,\bm{p},\bm{u})
&\coloneqq&
\mathbb{E}\Big [\prod_{j=1}^ds_j^{X_j}
M_{\bm{n}}(\bm{X};\alpha,\theta,\bm{p},\bm{u})
\Big ]
\nonumber \\
&=& \frac{\Gamma (\alpha + |\bm{n}|)}
{\Gamma (\alpha)n_0!\cdots n_{d-1}!}
\big [1 - \theta (S_0-1)\big ]^{-(\alpha + |\bm{n}|)}
\nonumber \\
&&\times (1+\theta)^{n_0}\big (1-S_0\big )^{n_0}\theta^{|\bm{n}|-n_0}
%\nonumber \\
%&&\times
\prod_{j=1}^{d-1}S_j^{n_j}.
\label{transform:1}
\end{eqnarray}
The transform of the 1-dimensional Meixner polynomials is
\begin{equation}
M_n^*(s;\alpha,\kappa) := \mathbb{E}\big [s^XM_n(X;\alpha,\kappa)\big ] 
= (1+\theta)^n(1-s)^n(1-\theta(s-1))^{-(\alpha + n)}.
\label{onedimMtrans}
\end{equation}
\end{cor}
The proofs in Theorem \ref{Def:M} and Corollary \ref{Corr:M} follow those for the multivariate Poisson-Charlier polynomials so are omitted for brevity. 
\begin{rem}
The multivariate Poisson-Charlier polynomials can be obtained as a limit from the multivariate Meixner polynomials. In the generating function (\ref{Meixner:6}) set $w_0=w_0^\prime\theta/|\bm{\mu}|$,
and $w_j^\prime  = w_j/|\bm{\mu}|$, $j=1,\ldots, d-1$. Now let $\alpha \to \infty$, $\theta \to 0$ with $\alpha\theta \to |\bm{\mu}|$. The generating function for the multivariate Meixner polynomials (\ref{Meixner:6}) converges to the generating function for the multivariate Poisson polynomials (\ref{MVPC:1}) with dummy variables $\bm{w}^\prime$. Therefore
\[
M_{\bm{n}}(\bm{X};\alpha,\theta,\bm{p},\bm{u})\theta^{n_0}|\bm{\mu}|^{-|\bm{n}|} \to C_{\bm{n}}(\bm{X};\bm{\mu},\bm{u}).
\]

\end{rem}

\subsection{Multivariate negative binomial Lancaster expansions}
A characterization of bivariate multivariate negative binomial distributions whose eigenfunctions are the multivariate Meixner polynomials is more difficult than the Poisson and Normal cases.
\begin{thm}\label{MP:0}Let $\bm{u}$ be an orthogonal basis on $\bm{p}=\bm{\mu}/|\bm{\mu}|$ with (\ref{basic:0}) holding. Suppose the basis satisfies the hypergroup property (\ref{dhyper:1}).
\medskip

\noindent
Then
\begin{eqnarray}
&&P(\bm{x};\alpha,\theta,\bm{p})
P(\bm{y};\alpha,\theta,\bm{p})
\nonumber \\
&&~\times\Big \{
1 + \sum_{|\bm{n}|\geq 1}\rho_{\bm{n}}h_{\bm{n}}(\alpha,\theta,\bm{u})
M_{\bm{n}}(\bm{x};\alpha,\theta,\bm{p},\bm{u})
M_{\bm{n}}(\bm{y};\alpha,\theta,\bm{p},\bm{u})
\Big \},
%\nonumber\\
\label{HM:0}
\end{eqnarray}
with $\{\rho_{\bm{n}}\}$ not depending on $\alpha$, is non-negative and thus a proper bivariate negative binomial distribution for all $\alpha >0$ if and only if
\begin{eqnarray}
\rho_{\bm{n}} = 
\mathbb{E}_{\bm{\xi}}\Big [
\prod_{j=0}^{d-1}\rho_j(\bm{\xi})^{n_j}\Big ],
\label{HM:1}
\end{eqnarray}
where
$\rho_j(\bm{\xi})=\sum_{i=1}^{d-1}u_i^{(j)}\xi_i$
and $\bm{\xi}$ is a random vector in the unit simplex satisfying a set of conditions.
Recall that 
\[
p_{ij}(\bm{\omega}) = p_ip_j\sum_{l=1}^d\omega_l\mathfrak{s}(i,j,l).
\]
where $\bm{\omega} = \bm{\xi}/|\bm{\xi}|$.
For non-random $\bm{\xi}$ the conditions are 
\begin{equation}
p_{ij}(\bm{\omega}) \geq \frac{\theta(1-|\bm{\xi}|)}{1 + \theta(1-|\bm{\xi}|)} p_ip_j,\>i,j = 1,\ldots, d.
\label{conditions:0}
\end{equation}
Sufficient conditions when $\bm{\xi}$ is random are also (\ref{conditions:0}) however we do not have a proof of the necessity of (\ref{conditions:0}) then.
\end{thm}
\begin{proof}
The proof of the necessity of (\ref{HM:1}) follows that in the Poisson case in Theorem \ref{HG:2}.
\medskip
The potential \emph{pgf} of $\bm{Y}$ given $\bm{X} = \bm{x}$ in (\ref{HM:0}) is
\begin{eqnarray}
&&\mathbb{E}\Big [\prod_{j=1}^d t_j^{Y_j}\mid \bm{X}=\bm{x}\Big ]
\nonumber \\
&&~=
\mathbb{E}_{\bm{\xi}}\sum_{|\bm{n}|\geq 0}\rho_{|\bm{n}|}h_{\bm{n}}(\alpha,\theta,\bm{u})
M_{\bm{n}}(\bm{x};\alpha,\theta,\bm{p},\bm{u})
M^*_{\bm{n}}(\bm{t};\alpha,\theta,\bm{p},\bm{u})
\nonumber \\
&&~= 
\big (1-\theta(T_0-1)\big )^{-\alpha}
\mathbb{E}_{\bm{\xi}}\Bigg [
\Bigg ( 1 - \frac{\theta |\bm{\xi}|(1-T_0)}{1-\theta (T_0-1)} \Bigg )^{-(\alpha + |\bm{x}|)} 
\nonumber \\
&&~~~\times 
\prod_{i=1}^d
\Bigg (1 + \frac{-(1+\theta)|\bm{\xi}|(1-T_0) + \sum_{j=1}^{d-1}\rho_j(\bm{\xi})a_j^{-1}T_j}{1-\theta(T_0-1)}\Bigg )^{x_i}\Bigg ]
\nonumber \\
&&~=
\mathbb{E}_{\bm{\xi}}\Bigg [\big (1-\theta(1-|\bm{\xi}|)(T_0-1)\big )^{-(\alpha + |\bm{x}|)}
\nonumber \\
&&~\times 
\prod_{i=1}^d
\Big ( 1 - (\theta - |\bm{\xi}|(1+\theta))(T_0-1) + |\bm{\xi}|\sum_{j=1}^d(p_{j\mid i}(\bm{\omega}) - p_j)t_j\Big )^{x_i}\Bigg ].~~~
\label{HM:2}
\end{eqnarray}
If $\bm{\xi}$ is non-random a necessary and sufficient condition that (\ref{HM:2}) be a \emph{pgf} for all $\bm{x} \in \mathbb{Z}_+^d$ and $\alpha > 0$ is that for each $i=1,\ldots ,d$ 
\begin{eqnarray}
&&\frac{
1 - (\theta - |\bm{\xi}|(1+\theta))(T_0-1) + |\bm{\xi}|\sum_{j=1}^d(p_{j\mid i}(\bm{\omega}) - p_j)t_j
}
{
1 - (\theta - |\bm{\xi}|(1+\theta))(T_0-1)
}
\nonumber \\
&&~~= 
1 + \frac{|\bm{\xi}|}{1+\theta(1-|\bm{\xi}|)}\cdot \frac{\sum_{j=1}^dp_{j\mid i}(\bm{\omega})t_j - 1}{1 - \kappa T_0}
\label{HM:3}
\end{eqnarray}
are \emph{pgf}s, where 
\[
\beta = \frac{|\bm{\xi}|}{1+\theta(1-|\bm{\xi}|)},\> \kappa = \frac{\theta (1-|\bm{\xi}|)}{1 + \theta (1-|\bm{\xi}|)}.
\]
The constant term in (\ref{HM:3}) is equal to
\[
\frac{(1-|\bm{\xi}|)(1+\theta)}{1+\theta(1-|\bm{\xi}|)} = \frac{1+\theta}{\theta}\kappa \geq 0.
\]
 The coefficient of $\prod_{j=1}^dt_j^{n_j}$ for $|\bm{n}| > 0$ in (\ref{HM:3}) is $\beta$ times
\begin{eqnarray}
&&\kappa^{|{\bm{n}}|-1}
\sum_{j=1}^dp_{j\mid i}(\bm{\omega})
{ |\bm{n}| - 1\choose \bm{n}-e_j}\prod_{l=1}^dp_l^{n_l-\delta_{lj}}
%\nonumber \\
%&&
-\kappa^{|\bm{n}|}{ |\bm{n}| - 1\choose \bm{n}}\prod_{l=1}^dp_l^{n_l}
\nonumber \\
&&=
\kappa^{|\bm{n}|}{ |\bm{n}|\choose \bm{n}}\prod_{l=1}^dp_l^{n_l}
\Big \{\kappa^{-1}\sum_{j=1}^dn_jp_{j\mid i}(\bm{\omega})/p_j - 1 \Big \}.
\label{HM:4}
\end{eqnarray}
The term in the brackets in (\ref{HM:4}) is non-negative for all $\bm{n}\in \mathbb{Z}_+^d$ if and only if (\ref{conditions:0}) holds.
We cannot argue the necessity in the same way if $\bm{\xi}$ is random.
\end{proof}
\begin{rem}
A sufficient condition for (\ref{conditions:0}) to hold is that for
$i,j=1,\ldots ,d$
\begin{equation}
p_{ij}(\bm{\omega}) \geq \frac{\theta}{1 + \theta} p_ip_j.
\label{extracond:0}
\end{equation}
\end{rem}
\begin{rem}
The joint \emph{pgf} of $(\bm{X},\bm{Y})$
with \emph{pgf} (\ref{HM:0}) is
\begin{eqnarray}
&&\mathbb{E}_{\bm{\xi}}\Big \{(1-\theta (S_0-1))(1 - \theta (T_0-1))
\nonumber \\%|\bm{\xi}|
&&~~ - |\bm{\xi}|\theta(1+\theta)(1-S_0)(1-T_0) 
- \theta|\bm{\xi}|\sum_{j=1}^{d}(p_{ij}(\omega)-p_ip_j)s_it_j\Big \}^{-\alpha}
\label{jt:0}
\end{eqnarray}
Three particular cases are:
\begin{enumerate}
\item
$|\bm{\xi|} = 0$; then $\bm{X}$ and $\bm{Y}$ are independent.
\item
$\bm{|\xi|} = 1$; then (\ref{jt:0}) is equal to 
%\begin{equation}
$
\mathbb{E}_{\bm{\omega}}
\big (1 + \theta - \theta \sum_{i,j=1}^dp_{ij}(\bm{\omega})s_it_j\big )^{-\alpha}
$
%\label{jt:1}
%\end{equation}
which corresponds to a bivariate multinomial distribution where the number of trials $N$ has a 1-dimensional negative binomial distribution. The conditions (\ref{extracond:0}) are then always satisfied assuming that $\bm{u}$ has the hypergroup property.
\item
$p_{ij}(\bm{\omega}) = p_ip_j$, $i,j=1,\ldots ,d$; then $\bm{X}$ and $\bm{Y}$ are conditionally independent given $(|\bm{X}|,|\bm{Y}|)$.
\end{enumerate}
\end{rem}
\begin{cor}
$(\bm{X},\bm{Y})$ has a distribution with a Lancaster expansion (\ref{HM:0}) if and only if, in the extreme points, the conditional \emph{pgf} of $\bm{Y}|\bm{X}=\bm{x}$ has the form
\begin{equation}
%&&
%\mathbb{E}_{(\beta,Q)}
\frac{(1-\kappa)^\alpha}{\big (1 - \kappa\sum_{j=1}^dp_jt_j\big )^{\alpha}}
%\nonumber \\
%&&~\times 
\prod_{i=1}^d\Bigg (1 - \beta 
+ \frac{\beta(1-\kappa)\sum_{j=1}^dq_{j\mid i}t_j}{1 - \kappa\sum_{j=1}^dp_jt_j}\Bigg )^{x_i},~~
\label{HM:5}
\end{equation}
where $0 \leq \beta \leq 1$, $\kappa = (1-\beta)\theta/(1+\theta)$ and
$Q=(q_{j\mid i})$ is a transition matrix, reversible with respect to $\bm{p}$, with eigenvalues $\bm{u}$. 
%$(\beta, Q)$ are random variables, not necessarily independent.
\end{cor}
\begin{proof}
In Theorem \ref{MP:0}, the conditional \emph{pgf} (\ref{HM:2}) 
and simplification (\ref{HM:3})
, make the 1-1 mapping 
\begin{equation}
\beta = \frac{|\bm{\xi}|}{1+\theta(1-|\bm{\xi}|)},\> 
\kappa = \frac{\theta(1-|\bm{\xi}|)}{1+\theta(1-|\bm{\xi}|)}
=\frac{(1-\beta)\theta}{1+\theta},\>
q_{j\mid i} = \frac{p_{j\mid i}(\omega)-\kappa p_j}{1-\kappa}.
\label{map:00}
\end{equation}
with $\bm{p}$ fixed.
$Q$ is a transition matrix since it has rows sums unity and is non-negative because of (\ref{extracond:0}).  In the inverse map to (\ref{map:00}) any $0 \leq \beta \leq 1$, $0 < \kappa < 1 - \beta$ and $Q$ determine
\begin{equation}
 \theta = \frac{\kappa}{1-\beta-\kappa},\> 
|\bm{\xi}| = \frac{\beta}{1-\kappa},\>
 p_{j\mid i}(\omega) = (1-\kappa)q_{j\mid i} + \kappa p_j
 \label{conditions:1}
 \end{equation}
 satisfying (\ref{conditions:0}).
 The calculation showing (\ref{HM:2}) is equivalent to (\ref{HM:5}) follows.
 Divide by $1+\theta(1-|\bm{\xi}|)$ within the brackets raized to the powers $x_i$ and express $p_{j\mid i}(\omega)$ as in (\ref{conditions:1}) to show that the \emph{pgf} is
\begin{eqnarray}
&&\big ( 1 - \theta(1-|\bm{\xi})(T_0-1)\big )^{-\alpha}
\Big (1 - \frac{\theta(1-|\bm{\xi}|)}
{1+\theta(1-|\bm{\xi}|)}
T_0 \Big )^{-|\bm{x}|}
\nonumber \\
&&\times \prod_{i=1}^d\Big (
1 - \frac{|\bm{\xi}|}{1+\theta(1-|\bm{\xi}|)}
 - \frac{\theta(1-|\bm{\xi}|)}{1+\theta(1-|\bm{\xi}|)}T_0
 \nonumber \\
&&~~+ \frac{|\bm{\xi}|(1-\kappa)}{1+\theta(1-|\bm{\xi}|)}\sum_{j=1}^dq_{j\mid i}t_j + \frac{\kappa|\bm{\xi}|}{1+\theta(1-|\bm{\xi}|)}T_0 \Big )^{x_i} 
\nonumber \\
&&=(1-\kappa)^{\alpha}(1-\kappa T_0)^{-\alpha}(1-\kappa T_0)^{-|\bm{x}|}
\nonumber \\
&&\times \Big ( (1-\beta)(1-\kappa T_0) + \beta (1-\kappa)\sum_{j=1}^dq_{j\mid i}t_j\Big )^{x_i}
\label{calc:10}
\end{eqnarray}
(\ref{HM:5}) now follows from (\ref{calc:10}) by dividing inside the bracketed terms by $1-\kappa T_0$.
\end{proof}
%%%
%%%
\begin{rem}An interpretation of the extreme points in the conditional distribution of $\bm{Y}\mid \bm{X}=\bm{x}$ with a \emph{pgf} 
(\ref{HM:5}) is as the transition functions in a discrete time Markov chain. Consider a population of individuals of $d$ types with configuration $\bm{x}$. In a transition each individual dies with probability $1-\beta$, or gives birth with probability $\beta$ according to the multivariate negative binomial \emph{pgf} $(1-\kappa)(1-\kappa \sum_1^dp_it_t)^{-1}$. Parents then change their types according to the transition matrix $Q$. Immigration occurs according to the \emph{pgf} $(1-\kappa)^\alpha(1-\kappa \sum_1^dp_it_t)^{-\alpha}$. The resulting population configuration is then $\bm{Y}$.
The stationary distribution of the Markov chain is multivariate Meixner with parameters $\alpha$, $\theta$, $\bm{p}$ specified by
(\ref{conditions:1}).
If $(\beta,Q)$ are random it may be possible that $Q$ has negative entries but the expected value of (\ref{HM:5}) is still a \emph{pgf}.
% The general transition functions are a mixture of extreme point transition functions. 
\end{rem}
\begin{rem}
A continuous time reversible Markov process $\{X(t)\}_{t\geq 0}$ with a stationary multivariate Meixner distribution and multivariate Meixner polynomial eigenfunctions can be constructed similarly to the multivariate Poisson in Remark \ref{ContinuousPoisson}; we omitt details of the calculation. Choose $\bm{\xi}$ so that (\ref{small:0}) - (\ref{calculation:0}) hold. 
$\{\bm{X}(t)\}_{t\geq 0}$ can be identified as a multitype birth and death process in a population of individuals of $d$ types; immigrants arrive at rates $\theta{\nu}\alpha p_j$; births occur to individuals at rates ${\nu}\theta p_j$, not depending on parental type; deaths occur to individuals at rate ${\nu}(1+\theta)$; and changes of type from $i$ to $j$ occur at rate $\widetilde{r}_{j\mid i} = {\bm{\gamma}}\widetilde{p}_{j\mid i} - \theta {\nu}p_j\geq 0$. The stationary distribution is multivariate Meixner. $|\bm{X}(t)|$ is a conventional birth and death process with birth rates when $|\bm{X}(t)|=|\bm{x}|$ of
$\theta {\nu}(\alpha + |\bm{x}|)$ and death rates ${\nu}(1+\theta )|\bm{x}|$ with a negative binomial stationary distribution. The spectral expansion for the transition functions of $|\bm{X}(t)|$ in terms of the 1-dimensional Meixner polynomials is well known, see for example \citet{S2000} p34.

\end{rem}
\section{Multivariate normal}
\subsection{Multivariate Hermite-Chebycheff polynomials}
A natural generalization of a product set of Hermite-Chebycheff polynomials, orthogonal on independent normal random variables, is to consider product sets of Hermite-Chebycheff polynomials in random variables which follow an orthogonal transformation in the original variables. The transformed variables are again independent. For clarity the details are shown in Theorem \ref{HCconst}. Of course the idea is straightforward and will be well known. \textcolor{black}{ A characterization of Lancaster distributions with these orthogonal polynomials as eigenfunctions in the following Section \ref{HCL} is new and not straightforward. }

Let $X$ be a $N(0,\tau)$ random variable.
The Hermite-Chebycheff polynomials, $\big \{H_n(X;\tau)\big \}_{n=0}^\infty$, orthogonal on $X$ which is $N(0,\tau)$  have a generating function
\begin{equation}
\sum_{n=0}^\infty H_n(X;\tau)\frac{z^n}{n!} = \exp \big  \{xz-\frac{1}{2}\tau z^2\big \}.
\label{HC:0}
\end{equation}
The orthogonality relationship is
\[
\mathbb{E}\big [H_m(X;\tau)H_n(X;\tau)\big ]
 = \delta_{mn}n!\tau^n.
\]
Let $\bm{X}=(X_1,\ldots,X_d)$ be independent normal random variables with variances $\bm{\tau} = (\tau_1,\ldots,\tau_d)$ and density
\begin{equation}
f(\bm{x};\bm{\tau}) = \prod_{i=1}^d\frac{1}{\sqrt{2\pi\tau_i}}
e^{-\frac{1}{2\tau_i}x_i^2},\> \bm{x}\in \mathbb{R}^d.
\label{NDen:0}
\end{equation}
 A set of multivariate Hermite-Chebycheff polynomials on $\bm{X}$ is the product set
$\bigotimes_{i=1}^d \big \{H_{n_i}(X_i;\tau_i)\big \}_{n_i=0}^\infty$.
Denote $\bm{p}=\bm{\tau}/|\bm{\tau}|$ and let $\bm{u}$
 be an orthogonal basis on 
$\bm{p}$, satisfying (\ref{basic:0}). Another set is obtained by considering the mapping 
\begin{equation}
\bm{X} \to \widehat{\bm{X}}=\Big (\sum_{i=1}^du_i^{(j)}X_i\Big )_{j=0}^{d-1}
\label{map:0}
\end{equation}
which preserves normality and independence because
\[
\text{Cov}\big (\widehat{X}_j,\widehat{X}_k\big )
= \sum_{i=1}^du_i^{(j)}u_i^{(k)}\tau_i = |\bm{\tau}|\delta_{jk}a_k,
\>j,k=0,1,\ldots ,d-1. 
\]
\begin{thm}\label{HCconst}
The set of multivariate Hermite-Chebycheff polynomials associated with the mapping (\ref{map:0}) is
\begin{equation}
H_{\bm{n}}(\bm{X};\bm{\tau},\bm{u}) = 
\prod_{j=0}^{d-1}H_{n_j}\big (\widehat{X}_j;|\bm{\tau}| a_j\big ),
\label{HC:1}
\end{equation}
where $\bm{n}_1 = (n_1,\ldots, n_{d-1})$.\\
The orthogonality relationship is
\begin{equation}
\mathbb{E}\Big [H_{\bm{m}}(\bm{X};\bm{\tau},\bm{u})H_{\bm{n}}(\bm{X};\bm{\tau},\bm{u})\Big ]
%&=& 
=
\delta_{\bm{m}\bm{n}}|\bm{\tau}|^{|\bm{n}|}\prod_{j=0}^{d-1}a_j^{n_j}
%\nonumber \\
%&\coloneqq&  
\coloneqq \delta_{\bm{m}\bm{n}}h_{\bm{n}}(|\bm{\tau}|,\bm{u})^{-1}.
\label{HCh:0}
\end{equation}
A generating function for the polynomials (\ref{HC:1}) is
\begin{eqnarray}
G_{\text{HC}}(\bm{X},\bm{w},\bm{u}) &=& 
\sum_{\bm{n}\geq 0}H_{\bm{n}}(\bm{X};\bm{\tau},\bm{u})
\prod_{j=0}^{d-1}\frac{w_j^{n_j}}{n_j!}
\nonumber \\
&=& \exp \Bigg \{\sum_{j=0}^{d-1}\widehat{X}_jw_j
- \frac{1}{2}|\bm{\tau}| \sum_{j=0}^{d-1}a_jw_j^2\Bigg \}.
\label{HC:2}
\end{eqnarray}
\end{thm}
\begin{rem}
The multivariate Poisson-Charlier polynomials tend to the multivariate Hermite-Chebycheff polynomials as the elements of $\bm{\mu}$ tend to infinity. This is to be expected as $|\bm{\mu}|^{1/2}(\bm{X}-\bm{\mu})$ converges to a normal random vector $\bm{Z}$ with independent entries and variances $\bm{p}$.
\begin{equation}
n_0!\cdots n_{d-1}!|\bm{\mu}|^{|\bm{n}|/2}(-1)^{n_0}
C_{\bm{n}}(|\bm{\mu}|^{1/2}\bm{Z}+\bm{\mu};\bm{\mu},\bm{u})
\to H_{\bm{n}}(\bm{Z};\bm{p},\bm{u}).
\label{PC_H:0}
\end{equation}
The limit (\ref{PC_H:0}) is seen from a straightforward generating function argument, whose proof is omitted.
%\begin{eqnarray*}
%&&\log G_{\text{PC}}(|\bm{\mu}|^{1/2}\bm{Z}+\bm{\mu};|\bm{\mu}|^{1/2}\bm{w},\bm{u})
%\nonumber \\
%&&=|\bm{\mu}|^{1/2}w_0 +
%\sum_{i=1}^d (|\bm{\mu}|^{1/2}Z_i+\mu_i)
%\log \big (1 - |\bm{\mu}|^{-1/2}w_0 + |\bm{\mu}|^{-1/2}
%\sum_{j=1}^{d-1}u_i^{(j)}w_j\big )
%\nonumber \\
%&&=
%-|\bm{Z}|w_0 + \sum_{i=1}^d\sum_{j=1}^{d-1}Z_iu_i^{(j)}w_j
%- \frac{1}{2}\sum_{i=1}^dp_i\Big (-w_0 +\sum_{j=1}^{d-1}u_i^{(j)}w_j\Big )^2 + o(|\bm{\mu}|^{-1/2})
%\nonumber \\
%&&\to-\widehat{Z}_0w_0+ \sum_{j=1}^{d-1}\widehat{Z}_jw_j
%-\frac{1}{2}\sum_{j=0}^{d-1}a_jw_j^2.
%\end{eqnarray*}

The multivariate Meixner polynomials tend to the multivariate Hermite-Chebycheff polynomials as $\alpha$ tends to infinity. $\alpha^{-1/2}(\bm{Z} - \alpha\theta \bm{p})$ has a limit normal distribution as $\alpha \to \infty$ with mean $\bm{0}$ and covariance matrix 
$\theta^2\bm{p}\bm{p}^T + \theta \text{~diag}(\bm{p})$. The variables in the limit are not independent, however $\widehat{Z}_j = \sum_{i=1}^dZ_iu_i^{(j)}$, $j=0,\ldots,d-1$ are independent with $\text{Var}(\widehat{Z}_0) = \theta(1+\theta)$ and $\text{Var}(\widehat{Z}_j) = \theta a_j$, $j \geq 1$.
\begin{equation}
n_0!\cdots n_{d-1}!\theta^{n_0}M_{\bm{n}}(\alpha^{1/2}\bm{Z}+\alpha\theta \bm{p};\alpha,\bm{p},\bm{u}) \to H^\prime_{\bm{n}}(\bm{Z};\bm{p},\bm{u}),
\label{MtoHC}
\end{equation}
where $H_{\bm{n}}^\prime$ is similarly defined to $H_{\bm{n}}$ in (\ref{HC:1}), but with $a_0=\theta(1+\theta)$ instead of $a_0=1$.
A generating function convergence shows (\ref{MtoHC}).
%, with a similar calculation to the Poisson case, is
%\begin{eqnarray*}
%&&\log G_{\text{M}}(\alpha^{1/2}+\alpha \theta \bm{p},\alpha^{-1/2}\bm{w},\alpha,\theta,u) 
%\nonumber \\
%&&\to -\widehat{Z}_0\theta^{-1}w_0 + \sum_{j=1}^{d-1}\widehat{Z}_jw_j - \frac{1}{2}(1+\theta)\theta^{-1}w_0^2 - \frac{1}{2}\theta\sum_{j=1}^{d-1}w_j^2.
%\end{eqnarray*}
\end{rem}
%
%
%\begin{rem} 
%\begin{equation}
%G_{\text{HC}}(\bm{x},\bm{e}_0,\bm{u}) = \exp \Big \{|\bm{X}|w_0 - \frac{1}{2}|\bm{\tau}| w_0^2\Big \}.
%\label{HC:3}
%\end{equation}
%generates Hermite-Chebycheff polynomials in the sum $|\bm{X}|$.
%\end{rem}
%
\begin{cor}
The transform
\begin{eqnarray}
 H^*_{\bm{n}}(\bm{s};\bm{\tau},\bm{u})
 &\coloneqq &
\mathbb{E}\Big [\prod_{i=1}^d\exp \Big \{\phi_i{X_i}\Big \}H_{\bm{n}}(\bm{X};\bm{\tau},\bm{u})
\Big ]
\nonumber \\
&=& 
\exp \Bigg \{\sum_{i=1}^{d}\tau_i\phi_i^2 \Bigg \}
\prod_{j=0}^{d-1}\Big (\sum_{i=1}^d\tau_iu_i^{(j)}\phi_i\Big )^{n_j}.
\label{HCtransform}
\end{eqnarray}
\end{cor}
\subsection{Multivariate normal Lancaster expansions}\label{HCL}
Lancaster distributions with $d$-dimensional normal marginals are related to classical canonical correlation concepts, see, for example \citet{L1969}, Chapter X, or the many books on multivariate analysis.
\textcolor{black}{
A characterization of bivariate normal distributions by \citet{SB1967}
which have a Lancaster expansion
\begin{equation}
\frac{1}{2\pi}e^{-\frac{1}{2}(x^2+y^2)}
\Bigg \{1 + \sum_{n=1}^\infty \rho_n(n!)^{-1}H_n(x;1)H_n(y;1)\Bigg \},\>-\infty < x,y < \infty
\label{sb:2}
\end{equation}
is that
$
\rho_n = \int_{-1}^1z^n\varphi (dz)
$.
 The bivariate distribution (\ref{sb:2}) is constructed from a usual bivariate normal with a random correlation coefficient with distribution $\varphi$. In the multivariate case the novelty is a characterization in the next theorem that $(\bm{X},\bm{Y})$ has a bivariate multivariate normal distribution with a \emph{random} cross correlation matrix which has a diagonal expansion (\ref{cov:0}) in terms of the basis $\bm{u}$. This is a significant characterization which extends that of \citet{SB1967} and not just a simple extension of canonical correlation theory.
}
\begin{thm} \label{theorem:6}
 Let $\{u^{(j)}\}_{j=0}^{d-1}$ be an orthogonal basis on $\bm{p}=\bm{\tau}/|\bm{\tau}|$ with (\ref{basic:0}) holding.
% 
% scaled so $u_d^{(j)} = 1$, $j=0,\ldots, d$, such that 
%$\sum_{i=1}^du_i^{(j)}u_i^{(k)}p_i=\delta_{jk}b_j^{-1}$, $j,k=0,\ldots,d-1$ %and
%$\mathfrak{s}(j,k,l) =\sum_{r=1}^db_ru_j^{(r)}u_k^{(r)}u_l^{(r)},%\>j,k,l=1,\ldots,d$.
Then
\begin{equation}
f(\bm{x},\bm{\tau})f(\bm{y},\bm{\tau})\Big \{
1 + \sum_{|\bm{n}|\geq 1}\rho_{\bm{n}}h_{\bm{n}}(|\bm{\tau}|,\bm{u})
H_{\bm{n}}(\bm{x};\bm{\tau},\bm{u})
H_{\bm{n}}(\bm{y};\bm{\tau},\bm{u}) \Big \}
\label{HHH:0}
\end{equation}
is non-negative and thus a proper bivariate normal distribution
if and only if
\begin{eqnarray}
\rho_{\bm{n}} = 
\mathbb{E}\Big [
\prod_{j=0}^{d-1}\xi_j^{n_j}\Big ],
\label{HHH:1}
\end{eqnarray}
where $\bm{\xi}$ is a random vector with elements in $[-1,1]$.
An equivalent statement to the characterization (\ref{HHH:1}) is that, conditional on $\bm{\xi}$, $(\bm{X},\bm{Y})$ is a bivariate normal pair of random vectors, independent within vectors, with means $(\bm{0},\bm{0})$, variances $(\bm{\tau},\bm{\tau})$ and cross covariances $V(\bm{\xi}) = \big (v_{ij}(\xi)\big )$, where 
\begin{equation}
v_{ij}(\bm{\xi}) = \text{Cov}(X_i,Y_j\mid \bm{\xi}) = |\bm{\tau}|p_ip_j
\sum_{l=0}^{d-1}\xi_la_l^{-1}u_i^{(l)}u_j^{(l)}.
\label{cov:0}
\end{equation}
\end{thm}
\begin{proof} 
\emph{Necessity.} Suppose (\ref{HHH:0}) holds. Let $(\bm{X},\bm{Y})$ be a pair of normal random vectors with distribution (\ref{HHH:1}). Then
\[
\rho_{\bm{n}}H_{\bm{n}}(\bm{x};\bm{\tau},\bm{u})
= \mathbb{E}\Big [H_{\bm{n}}(\bm{Y};\bm{\tau},\bm{u}) \mid \bm{X}=\bm{x}\Big ]
\]
and there is only one leading term of degree $|\bm{n}|$, proportional to 
\[
m(\bm{n},\widehat{\bm{x}},\bm{u})
= \prod_{j=0}^{d-1}\widehat{x}_j^{n_j}
\]
on the left side. Rearranging
\begin{equation}
\mathbb{E}\Big [m(\bm{n},\widehat{\bm{Y}},\bm{u})\mid \bm{X}=x\Big ]
=\rho_{\bm{n}}m(\bm{n},\widehat{\bm{x}},\bm{u})
+ R_{|\bm{n}|-1}(\widehat{\bm{x}})
\label{HHH:3}
\end{equation}
where $R_{|\bm{n}|-1}(\widehat{\bm{x}})$ is a polynomial of degree 
$|\bm{n}|-1$ in $\widehat{\bm{x}}$. Divide (\ref{HHH:3}) by
$m(\bm{n},\widehat{\bm{x}},\bm{u}) $ and let $\widehat{x}_j \to \infty$, $j=0\ldots ,d-1$. 
Let $\bm{\xi}$ be a random variable with the limit distribution of 
$\big (\widehat{Y_j}/\widehat{x_j}\big )$ given 
$\widehat{\bm{X}} = \widehat{\bm{x}}$. Then (\ref{HHH:1}) holds.
\medskip

\noindent
\emph{Sufficiency} Suppose that (\ref{HHH:1}) holds. 
Let $\rho_r = \sum_{i=1}^du_i^{(r)}\xi_i$. The transform of a potential exchangeable bivariate distribution $(\bm{X},\bm{Y})$ formed from (\ref{HHH:0}) is
\begin{eqnarray}
&&\mathbb{E}\Big [\prod_{i,j=1}^d\exp \Big \{\phi_i X_i + \psi_j Y_j\Big \}\Big ]
\nonumber \\
&&~=\sum_{\bm{n}:|\bm{n}| \geq 0}
\rho_{\bm{n}}h_{\bm{n}}(|\bm{\tau}|,\bm{u})
H^*_{\bm{n}}(\bm{s};\bm{\tau},\bm{u})
H^*_{\bm{n}}(\bm{t};\bm{\tau},\bm{u})
~~~~
\nonumber \\
&&~=\exp \Big \{\sum_{i=1}^d\tau_i(\phi_i^2+\psi_i^2) \Big \}
\nonumber \\
&&~~\times
\mathbb{E}\Big [\exp \Big \{
\sum_{r=0}^{d-1}|\bm{\tau}| a_r^{-1}\rho_r
\Big (\sum_{i=1}^dp_i\phi_iu_i^{(r)}\Big )
\Big (\sum_{j=1}^dp_j\psi_ju_j^{(r)}\Big )\Big \}\Big ]
\nonumber \\
&&~=
\exp \Big \{\sum_{i=1}^d\tau_i(\phi_i^2+\psi_i^2) \Big \}
\nonumber \\
&&~~\times
\mathbb{E}\Big [\exp \Big \{
|\bm{\tau}|\sum_{i,j=1}^dp_ip_j\sum_{r=0}^{d-1}\xi_ra_r^{-1}u^{(r)}_iu_j^{(r)}\phi_i\psi_j
\Big \}\Big ]~~~
\label{HHH:4}
\end{eqnarray}
which is a proper moment generating function with the covariance structure 
(\ref{cov:0}). To see this note that the conditional distribution of
$(\widehat{\bm{X}},\widehat{\bm{Y}})$ 
consists of independent bivariate normal pairs $\{(\widehat{X}_i,\widehat{Y}_i)\}_{i=1}^d$.
The covariance structure is
\begin{eqnarray*}
\text{Cov}(\widehat{X}_r,\widehat{Y}_s\mid\bm{\xi}) &=& \sum_{i,j=1}^d
u^{(r)}_iu_j^{(s)}\text{Cov}(X_i,Y_j\mid \bm{\xi})
\nonumber \\
&=& \sum_{i,j=1}^du^{(r)}_iu_j^{(s)}
|\bm{\tau}|p_ip_j\sum_{l=0}^{d-1}\xi_la_l^{-1}u^{(l)}_iu_j^{(l)}
\nonumber \\
&=& \delta_{rs}|\bm{\tau}|\xi_ra_r.
\end{eqnarray*}
That is
\begin{equation*}
\text{Var}(\widehat{X}_r\mid \bm{\xi}) = \text{Var}(\widehat{Y}_r\mid \bm{\xi})
%=\sum_{i=1}^d{u^{(r)}_i}^2\tau_i 
= |\bm{\tau}|a_r,\>
\text{Corr}(\widehat{X}_r,\widehat{Y}_r\mid \bm{\xi})
=\xi_r.
\end{equation*}

\end{proof}
\begin{rem} $\bm{X}$ and $\bm{Y}$ are independent if
$\rho_{\bm{n}} = 0$ for all $|\bm{n}| \geq 1$. 
$\bm{X}=\bm{Y}$ if $\rho_{\bm{n}} = 1$ for $|\bm{n}| \geq 1$,
when $\xi_0=\cdots = \xi_{d-1}=1$. These two cases are true from general theory, and also follow from evaluating (\ref{HHH:4}).
\end{rem}
\begin{rem}
Theorem \ref{theorem:6} includes the classical case of canonical correlation of a pair of exchangeable normal vectors $(\bm{X},\bm{Y})$
when $\bm{\xi}$ is constant. There is an insistance that the cross correlation matrix have an expansion (\ref{cov:0}). This is like a Lancaster expansion, but non-negativity is not required.

Theorem \ref{theorem:6} is a new extension of a canonical correlation expansion for the distribution of $(\bm{X},\bm{Y})$ with the particular transformation to  $(\widehat{\bm{X}},\widehat{\bm{Y}})$, with a \emph{random} cross covariance matrix. 
%The characterization is similar to a canonical correlation characterization, with a random covariance structure such that the marginal vectors are normal. 
The joint moment generating function of $(\bm{X},\bm{Y})$ is
\begin{equation}
\mathbb{E}_{\bm{\xi}}\Big [\exp \Big \{\frac{1}{2}\bm{s}^T\bm{s} + \frac{1}{2}\bm{t}^T\bm{t}
+\bm{s}^TV(\bm{\xi})\bm{t})\Big \}\Big ]
\label{mgf:0}
\end{equation}
where $V(\bm{\xi})$ is defined in (\ref{cov:0}) and the conditional moment generating function of $\bm{Y}\mid \bm{X}$ is
\begin{equation}
\mathbb{E}_{\bm{\xi}}\Big [
\exp \Big \{\frac{1}{2}\bm{t}^T(I-V(\bm{\xi})^2)\bm{t} + \bm{t}^TV(\bm{\xi})\bm{X}\Big \}\Big ].
\label{mgf:1}
\end{equation}
\end{rem}
\begin{rem}
The conditional distribution of $\bm{Y}\mid \bm{X}=\bm{x}$ has an interpretation as a transition function in a discrete time Markov chain, where transitions are made from $\bm{x}$ to $Y$ which is normal with mean $V(\bm{\xi})\bm{x}$ and covariance matrix $I-V(\bm{\xi})^2$, with the cross covariance matrices $V(\bm{\xi})$  identically distributed at each epoch.
\end{rem}

\section{\textcolor{black}{References}}


\begin{thebibliography}{99}
\bibitem[Aitken and Gonin(1935)]{AG1935}
{\sc Aitken, A. C. and Gonin, H. T.} (1935)
On fourfold sampling with and without replacement. 
\emph{Proc. Roy. Soc. Edinb.} {\bf 55} 114--125.


\bibitem[Bakry and Huet(2008)]{BH2008} 
{\sc Bakry, D. and Huet, N.} (2006) The hypergroup property and representation of Markov Kernels. \emph{S{\'e}minare de Probabiliti{\'e}s XLI, Lecture notes in Mathematics}, Vol 1934, 295--347, Springer.

\bibitem[Bochner(1954)]{B1954}
{\sc Bochner, S.} (1954).
Positive zonal functions on spheres, \emph{Proc. Nat. Acad. Sci. USA} \textbf{40}  1141--1147.

\bibitem[Diaconis and Griffiths(2012)]{DG2012} 
{\sc Diaconis, P. and Griffiths R. C.} (2012) Exchangeable pairs of Bernoulli random variables, Krawtchouk polynomials, and Ehrenfest urns. \emph{Aust. NZ J. Stat.} {\bf 54} 81--101.

\bibitem[Diaconis and Griffiths(2014)]{DG2014} 
{\sc Diaconis, P. and Griffiths R. C.} (2014)
An introduction to multivariate Krawtchouk polynomials and their applications.
\emph{Journal of Statistical Planning and Inference} {\bf 154} 39--53.

\bibitem[Eagleson(1964)]{E1964}
{\sc Eagleson, G. K.} (1964) Polynomial expansions of bivariate distributions. {\em Ann. Math. Statist.} {\bf 35} 1208--1215.

\bibitem[Eagleson(1969)]{E1969}
{\sc Eagleson, G. K.} (1969) A characterization theorem for positive definite sequences on the Krawtchouk polynomials. \emph{Austral. J. Statist.} {\bf 11} 29--38.

\bibitem[Genest et. al.(2013)]{GVZ2013}
{\sc Genest, V. X., Vinet, L. and Zhedanov, A.} (2013). The multivariate Krawtchouk polynomials as matrix elements of the rotation group representations on oscillator states. 
\emph{J. Phys A: Math. Theor.} {\bf 46} 505203. 

\bibitem[Genest et. al.(2014)]{GMVZ2014}
{\sc Genest, V. X., Miki, H., Vinet, L., and Zhedanov, A.} (2014). The multivariate Meixner polynomials as matrix elements of SO (d, 1) representations on oscillator states. \emph{J. Phys A: Math. Theor.} {\bf 47} 045207.

\bibitem[Genest et. al.(2014a)]{GMVZ2014a}
{\sc Genest, V. X., Miki, H., Vinet, L., and Zhedanov, A.} (2014a). The multivariate Charlier polynomials as matrix elements of the Euclidean group representation on oscillator states. 
\emph{J. Phys A: Math. Theor.} {\bf 47} 215204.

\bibitem[Griffiths(1969)]{G1969}
{\sc Griffiths R. C.}(1969)
The canonical correlation coefficients of bivariate gamma distributions.
{\em Ann. Math. Statist.} {\bf 40} 1401--1408.


\bibitem[Griffiths(1971)]{G1971}
{\sc Griffiths, R. C.} (1971) 
Orthogonal polynomials on the multinomial distribution. {\em Austral. J. Statist.} {\bf 13} 27--35. Corrigenda (1972) {\em Austral. J. Statist.} {\bf 14} 270.

\bibitem[Griffiths(1975)]{G1975}
{\sc Griffiths, R.} (1975). Orthogonal polynomials on the negative multinomial distribution. {\em J. Multivariate Anal.} {\bf 5} 271--277.

\bibitem[Griffiths(2009)]{G2009}
{\sc Griffiths, B. [R.C.]} (2009).
Stochastic Processes with Orthogonal Polynomial Eigenfunctions,
\emph{J. Comput. Appl. Math.} \textbf{23} 739--744.

\bibitem[Grunbaum and Rahman(2011)]{GR2011}
{\sc Grunbaum, F. and Rahman, M.} (2011) A system of multivariable Krawtchouk polynomials and a probabilistic application. \emph{SIGMA} {\bf 7} 119--136.

\bibitem[Hoare and Rahman(1983)]{HR1983}
{\sc Hoare, M. and Rahman, M.} (1983) Cumulative Bernoulli trials and Krawtchouk processes. {\em Stochastic Process. Appl.} {\bf 16} 113--139.

\bibitem[Hoare and Rahman(2008)]{HR2008}
{\sc Hoare, M., Rahman, M.} (2008). A probabilistic origin for a new class of bivariate polynomials. \emph{SIGMA} {\bf 4} 18 pages.

\bibitem[Iliev and Xu(2007)]{IX2007}
{\sc Iliev, P. and Xu, Y.} (2007). Discrete orthogonal polynomials and difference equations of several variables. \emph{Advances in Mathematics} {\bf 212} 1--36.

\bibitem[Iliev(2012)]{I2012} 
{\sc Iliev, P.} (2012) A Lie-theoretic interpretation of multivariate hypergeometric polynomials. {\em Compositio Mathematica} {\bf 148} 991--1002.

\bibitem[Iliev(2012a)]{I2012a}
{\sc Iliev, P.} (2012a). Meixner polynomials in several variables satisfying bispectral difference equations. \emph{Advances in Applied Mathematics} {\bf 49} 15--23.

\bibitem[Ismail(2005)]{I2005}
{\sc Ismail, M. E. H.} (2005) {\em Classical and Quantum Orthogonal Polynomials in one variable}, Volume 98 of {\em Encyclopedia of Mathematics and its Applications.} Cambridge: Cambridge University Press.



\bibitem[Koudou(1996)]{K1996}
{\sc Koudou  A. E.} (1996)
Probabiliti\'es de Lancaster,
{\em Exposition. Math.} {\bf 1} 247--275.

\bibitem[Koudou(1998)]{K1998}
{\sc Koudou A. E.} (1998)
Lancaster bivariate probability distributions with Poisson, Negative Binomial and Gamma margins,
{\em Sociedad de Estad{\'i}stica e Investigaci\'on Operativa Test}.
{\bf 7} 95--110.

\bibitem[Lancaster(1965)]{L1965}
{\sc Lancaster, H.} (1965) The Helmert matrices. \emph{Amer Math. Monthly} {\bf 72} 4--12.

\bibitem[Lancaster(1969)]{L1969}
{\sc Lancaster H.} (1969)
{\em The chi-squared distribution},
John Wiley \& Sons.


\bibitem[Lancaster(1975)]{L1975}
{\sc Lancaster, H.} (1975) Distributions in the Meixner Classes. \emph{J. Roy. Statist. Soc. Ser. B} {\bf 37} 434--443.



\bibitem[Meixner(1934)]{M1934}
{\sc Meixner, J.} (1934) Orthogonale Polynomsysteme mit einer besonderen Gestalt der erzeugenden Funktion. \emph{J. London Math. Soc.} {\bf 9} 6--13.


\bibitem[Mizukawa(2010)]{M2010} 
{\sc Mizukawa, H.} (2010) Finite Gelfand pair approaches for Ehrenfest diffusion model.
arXiv:1009.1205.


\bibitem[Mizukawa(2011)]{M2011}
{\sc Mizukawa, H.} (2011) Orthogonality relations for multivariate Krawtchouk polynomials.
arXiv:1009.1203.

\bibitem[Mizukawa and Tanaka(2004)]{MT2004}
{\sc Mizukawa, H. and Tanaka, H.} (2004). $(n+1,m+1)-$ hypergeometric functions associated to character algebras. {\em Proceedings of the American Mathematical Society} {\bf 132}, 2613--2618.


\bibitem[Sarmanov(1968)]{S1968}
{\sc Sarmanov, I. O.} (1968)
A generalized symmetric gamma-correlation,
{\em Dokl. Akad. Nauk SSSR} {\bf 179} (1968) 1279--1281.

\bibitem[Sarmanov and Bratoeva(1967)]{SB1967}
{\sc Sarmanov I. O. and Bratoeva, Z. N.} (1967) 
Probabilistic properties of bilinear expansions of Hermite polynomials.
{\em Theor. Probability Appl.} {\bf 12} 470--481.

\bibitem[Schoutens(2000)]{S2000}
{\sc Schoutens, W.}
{\em Stochastic processes and orthogonal polynomials.}
Lecture notes in mathematics 146 Springer-Verlag.

\textcolor{black}{
\bibitem[Tratnik(1989)]{T1989}
{\sc Tratnik, M. V.} (1989) Multivariable Meixner, Krawtchouk, and Meixner-Pollaczek polynomials. {\em J. Math. Phys.} {\bf 30} 2740--2749.
}

\textcolor{black}{
\bibitem[Tratnik(1991)]{T1991}
{\sc Tratnik, M. V.} (1991) Some multivariable orthogonal polynomials of the Askey tableau-discrete families. {\em J. Math. Phys.} {\bf 32} 2337--2342.
}




\bibitem[Xu(2014)]{X2013}
{\sc Xu, Y.} (2015)
Hahn, Jacobi, and Krawtchouk polynomials of several variables. {\em J. Approx. Theory.} {\bf 195} 19--42.



\bibitem[Zhou and Lange(2009)]{ZL2009} 
{\sc Zhou, H. and Lange, K.} (2009)
Composition Markov chains of multinomial type. {\em Adv. Appl. Probab.} {\bf 41} 270--291.

\end{thebibliography}
\end{document}